\documentclass[11pt]{amsart}
\usepackage{amssymb}
\usepackage{mathpazo} 
\usepackage{mathrsfs}
\usepackage{fullpage}
\usepackage{enumerate} % for controlable nested enumeration
\usepackage{verbatim} % for \begin{comment}
\usepackage{url}
\usepackage{bbm}

\title{The moduli scheme of affine spherical varieties with a free weight monoid}

\author{Paolo Bravi}
\address{Dipartimento di Matematica ``Guido Castelnuovo'', ``Sapienza'' Universit\`a di Roma, Piazzale 
Aldo Moro 5, 00185 Roma, Italy}
\email{bravi@mat.uniroma1.it}

\author{Bart Van Steirteghem}
\address{Department of Mathematics, Medgar Evers College - City University of New York, 1650 Bedford Ave., Brooklyn, NY 11225, USA}
\email{bartvs@mec.cuny.edu}

\numberwithin{equation}{section}

\newcommand{\inn}{\subset}
\newcommand{\isom}{\simeq}

\newcommand{\into}{\hookrightarrow}
\renewcommand{\>}{\rangle}
\newcommand{\<}{\langle}

\renewcommand{\epsilon}{\varepsilon}

\DeclareMathOperator{\Spec}{Spec}

\DeclareMathOperator{\Hom}{Hom}

\DeclareMathOperator{\supp}{supp}

\newcommand{\mf}{\mathfrak}

\newcommand{\fb}{\mf{b}}

\newcommand{\fg}{\mf{g}}

\newcommand{\fn}{\mf{n}}

\newcommand{\ft}{\mf{t}}

\renewcommand{\sl}{\mf{sl}}

\newcommand{\msf}{\mathsf}
\newcommand{\ssA}{\msf{A}}
\newcommand{\ssB}{\msf{B}}
\newcommand{\ssC}{\msf{C}}
\newcommand{\ssD}{\msf{D}}
\newcommand{\ssE}{\msf{E}}
\newcommand{\ssF}{\msf{F}}
\newcommand{\ssG}{\msf{G}}

\newcommand{\GL}{\mathrm{GL}}
\newcommand{\SL}{\mathrm{SL}}

\newcommand{\shN}{\mathcal{N}}

\newcommand{\N}{{\mathbb N}}

\newcommand{\Q}{{\mathbb Q}}

\newcommand{\Z}{{\mathbb Z}}

\newcommand{\tM}{\mathrm{M}}

\newcommand{\wl}{\Lambda}
\newcommand{\rl}{\Z\sr}
\newcommand{\dw}{\Lambda^+}
\newcommand{\sr}{S}  %simple roots
\newcommand{\pr}{R^+} % positive roots
\newcommand{\wm}{\Gamma}
\newcommand{\ms}{\mathrm{M}_{\wm}}
\newcommand{\hs}{\mathrm{H}_{\wm}}
\newcommand{\om}{\omega}
\newcommand{\Tad}{T_{\mathrm{ad}}}
\newcommand{\tHilb}{\mathrm{Hilb}}
\newcommand{\hwv}{x_0}
\newcommand{\Vgg}{(V/\fg\cdot\hwv)^{G_{\hwv}}}
\newcommand{\Vg}{V/\fg\cdot\hwv}

\renewcommand{\k}{\mathbbm{k}}
\newcommand{\rs}{R}
\newcommand{\tg}{\mathrm T_{X_0}}

\theoremstyle{plain}
\newtheorem{theorem}{Theorem}[section]
\newtheorem*{theorem*}{Theorem}
\newtheorem{lemma}[theorem]{Lemma}
\newtheorem{proposition}[theorem]{Proposition}
\newtheorem{corollary}[theorem]{Corollary}
\newtheorem{conjecture}[theorem]{Conjecture}

\theoremstyle{definition}
\newtheorem{definition}[theorem]{Definition}
\newtheorem{remark}[theorem]{Remark}

\begin{document}

\begin{abstract}
We study Alexeev and Brion's moduli scheme $\ms$ of affine spherical
varieties with weight monoid $\wm$ under the assumption that $\wm$ is
free. We describe the tangent space to $\ms$ at its `most degenerate
point' in terms of the combinatorial invariants of spherical
varieties and deduce that the irreducible components of $\ms$,
equipped with their reduced induced scheme structure, 
are affine spaces.   
\end{abstract}

\maketitle

\section{Introduction} \label{sec:introduction}
As part of the classification problem of algebraic varieties equipped with a group action, spherical varieties, which include symmetric, toric and flag varieties, have received considerable attention; see, e.g., \cite{brion-gensymm, knop-auto, luna-typeA, losev-uniqueness}. In~\cite{alexeev&brion-modaff}, V.~Alexeev and M.~Brion introduced an important new tool for the study of affine spherical varieties over an algebraically closed field $\k$ of characteristic $0$. 
We recall that an affine variety $X$ equipped with an action of a connected reductive group $G$ is called spherical if it is normal 
and its coordinate ring $\k[X]$ is multiplicity-free as a $G$-module. 
For such a variety a natural invariant, which completely describes the $G$-module structure of $\k[X]$, 
is its \textbf{weight monoid} $\wm(X)$. 
By definition, $\wm(X)$ is the set of isomorphism classes of irreducible representations of $G$ that occur in $\k[X]$. 
In view of the classification problem, we have the following natural question: how `good' an invariant is $\wm(X)$, or more explicitly: 
to what extent does $\wm(X)$ determine the multiplicative structure of $\k[X]$? 

Alexeev and Brion brought geometry to this question as follows. After choosing a Borel subgroup $B$ of $G$, and a maximal torus $T$ in $B$, we can identify $\wm(X)$ with a finitely generated submonoid of the monoid $\dw$ of dominant weights.  Let $\wm$ be another such submonoid of $\dw$ and put
\[V(\wm) = \oplus_{\lambda \in \wm} V(\lambda),\]
where we used $V(\lambda)$ for the irreducible $G$-module corresponding to $\lambda \in \dw$. 
Let $U$ be the unipotent radical of $B$ and let $V(\wm)^U$ be the subspace of $U$-invariants, which is also the space of highest weight vectors in $V(\wm)$. By choosing an isomorphism $V(\wm)^U \to \k[\wm]$ of $T$-modules, where $\k[\wm]$ is the semigroup ring associated to $\wm$, we equip $V(\wm)^U$ with a $T$-multiplication law. Alexeev and Brion's moduli scheme $\ms$ parametrizes the $G$-multiplication laws on $V(\wm)$ which extend the multiplication law on $V(\wm)^U$. For an introduction to this moduli scheme, we refer the reader to \cite[\S 4.3]{brion-ihs}.
Examples of $\ms$ have been computed in \cite{jansou-deformations, bravi&cupit, degensphermod}. 

Let $\wl$ be the weight lattice of $G$, that is, $\wl$ is the character group of $T$. Because $X$ is normal, its weight monoid $\wm(X)$ also satisfies the following equality in $\wl \otimes_\Z \Q$
\begin{equation}  \label{eq:3}
\wm(X) = \Z\wm(X) \cap \Q_{\ge 0}\wm.
\end{equation}
By definition, this makes $\wm(X)$ a \textbf{normal} submonoid of
$\dw$. 

In \cite{brion-ihs}, Brion conjectured that the irreducible components of $\ms$ are affine spaces. A precise version of this conjecture is the following.
\begin{conjecture} \label{conj:brion}
If $\wm$ is a normal submonoid of $\dw$, then the irreducible components of $\ms$, equipped with their reduced induced scheme structure, are affine spaces. 
\end{conjecture}
This conjecture was verified for free and $G$-saturated monoids of
dominant weights in \cite{bravi&cupit}. In fact, Bravi and
Cupit-Foutou proved that under these assumptions, $\ms$ is an affine
space. In \cite{degensphermod,dsmot-preprint} it is shown that $\ms$
is an affine space when $\wm$ is the weight monoid of a spherical
$G$-module. 
Luna provided the first non-irreducible example (unpublished):
for $G=\SL(2) \times \SL(2)$ and $\wm = \<2\om, 4\om + 2\om'\>$,
where $\om$ and $\om'$ are the fundamental weights of the two copies of $\SL(2)$,
the scheme $\ms$ is the union of two lines meeting in a point.
In this paper, we verify that Conjecture~\ref{conj:brion} holds
when $\wm$ is free.
\begin{theorem}[Corollary~\ref{cor:brionconj}]\label{thm:A} 
If $\wm$ is a free submonoid of $\dw$, then the irreducible components
of $\ms$, equipped with their reduced induced scheme structure, are affine spaces. 
\end{theorem}

The bulk of this paper is devoted to the description of the tangent space to $\ms$ at its `most degenerate point' $X_0$ in terms of certain combinatorial invariants, called N-spherical roots. To be more precise, we introduce some more terminology and recall some facts. If $X$ is an affine spherical $G$-variety $X$, then its \textbf{root monoid} $\mathscr M_X$ is the submonoid of $\wl$ generated by the set
\[\{\lambda + \mu - \nu \, |\, \lambda,\mu,\nu \in \dw \text{ such that }\<\k[X]_{(\lambda)}\cdot \k[X]_{(\mu)}\>_\k \cap \k[X]_{(\nu)} \neq 0\}.\]
Here $\k[X]_{(\lambda)}$ is the isotypic component of $\k[X]$ of type
$\lambda \in \dw$. Loosely speaking, $\mathscr M_X$ detects how far
the decomposition $\k[X]=\oplus_{\lambda \in \wm(X)}\k[X]_{(\lambda)}$
is from being a grading by $\wm(X)$. A deep result by
Knop~\cite[Theorem~1.3]{knop-auto} says that the saturation of $\mathscr M_X$, which is the intersection in $\wl \otimes_{\Z} \Q$ of the cone
\(\Q_{\ge 0} \mathscr M_X\) and the lattice $\Z \mathscr M_X$, is a
freely generated monoid. Its basis $\Sigma^N(X)$ is called the set of
\textbf{N-spherical roots} of $X$. By~\cite[Proposition~2.13]{alexeev&brion-modaff} a formal
 consequence of our theorem above is that if $X$ is an
 affine spherical $G$-variety with a free weight monoid, then its root
 monoid $\mathscr M_X$ is also free; see Corollary~\ref{cor:rootmonfree}.

In their seminal paper~\cite{alexeev&brion-modaff}, Alexeev and Brion
equipped $\ms$ with an action of the maximal torus $T$ of $G$. For
this action, $\ms$ has a unique closed orbit, which is a fixed point
$X_0$. Consequently, the tangent space $\tg \ms$ to $\ms$ at the
point $X_0$ is a finite-dimensional
$T$-module. We describe this tangent space as follows.

\begin{theorem}[Theorem~\ref{thm:TX0hs} and
  Corollary~\ref{cor:TX0hs}]\label{thm:B} 
If $\wm$ is a free submonoid of $\dw$ ,
then $\tg \ms$ is a multiplicity-free $T$-module, and $\gamma \in
\wl$ occurs as a weight in  $\tg \ms$ if and only if there exists
an affine spherical $G$-variety $X_{-\gamma}$ with weight monoid $\wm$
and $\Sigma^N(X_{-\gamma}) = \{-\gamma\}$. 
\end{theorem}
To prove this we first use the combinatorial
theory of spherical varieties~\cite{knop-lv, luna-typeA,
  losev-uniqueness} to combinatorially characterize the weights $\gamma$ for which such a variety $X_{-\gamma}$ exists; see Corollary~\ref{cor:indiv_Nadapt-spher-roots}.  Such a
characterization was sketched by Luna in 2005 in
an unpublished note. 

To prove Theorem~\ref{thm:A} we use Theorem~\ref{thm:B}:
since it is known that the irreducible components of $\ms$,
equipped with their reduced induced scheme structure, are affine spaces after normalization
(by \cite[Theorem~1.3]{knop-auto} and \cite[Corollary~2.14]{alexeev&brion-modaff}),
it is enough to show that they are smooth,
and this follows from our description of the tangent space to $\ms$ at $X_0$
(see Section~\ref{sec:components}).

\subsection*{Notation}
Except if explicitly stated otherwise, $\wm$ will be a \emph{free} submonoid of
$\dw$ with basis 
\(F = \{\lambda_1,\lambda_2, \ldots, \lambda_r\}.\) 
We will use $\sr$ for the set of simple roots of $G$ (associated to $B$ and $T$) 
and $\pr$ for the set of positive roots. 
The irreducible representation of $G$ associated to the
dominant weight $\lambda \in \dw$ is denoted by $V(\lambda)$ and we use
$v_{\lambda}$ for a highest weight vector in $V(\lambda)$.  We use
$\fg, \fb, \ft, \fn,$ etc. for the Lie algebra of $G,B,T, U,$ etc.,
respectively. When $\alpha$ is a root,
$X_{\alpha} \in \fg_{\alpha}$ is a root operator and $\alpha^{\vee}$
the coroot. 
When $\fg$ is simple, simple roots are denoted by $\alpha_1,\ldots,\alpha_n$ 
and numbered like by Bourbaki (see \cite{bourbaki-geadl47}), 
the corresponding fundamental weights are denoted by $\om_1,\ldots,\om_n$.

\subsection*{Acknowledgement}
The authors are grateful to the Institut Fourier for hosting them in the summer of 2011, when work on this project began. They also thank the Centro Internazionale per la Ricerca Matematica in Trento, as well as Friedrich Knop and the Emmy Noether Zentrum in Erlangen for their hospitality in the summers of 2012 and 2013, respectively. 
The second-named author received support from The City University of New York PSC-CUNY Research Award Program and from the National Science Foundation through grant DMS 1407394. 

The authors are grateful to Michel Brion for suggesting this
problem and for helpful discussions. They also thank Domingo Luna for discussions and suggestions
in the summer of 2011 and for sharing his working paper of 2005;  they were particularly helpful for Section~\ref{sec:spher-roots-adapt}. They would like to thank Jarod Alper for a clarifying exchange, summarized in Remark~\ref{rem:alper}, about scheme structures on irreducible components of affine schemes. It alerted them to a mistake in an earlier version of this paper.  

The authors thank the referee for several helpful suggestions, and in particular for providing an elementary proof of Proposition~\ref{prop:weightmonoid}(\ref{item:8}) and for correcting an error in an earlier version of the proof of (\ref{item:9}) and (\ref{item:10}) of the same proposition.

As this paper was being completed, R.~Avdeev and S.~Cupit-Foutou
announced that they had independently obtained similar results
\cite{avdeev&cupit-irrcomps-arxivv1}. 

\subsection*{Note added during review}
While this paper was under review, a second version of the preprint \cite{avdeev&cupit-irrcomps-arxivv1} was posted on the arXiv, in which Avdeev and Cupit-Foutou propose a proof of Conjecture~\ref{conj:brion} for all normal monoids $\wm$ and an example of a non-reduced moduli scheme $\ms$ (cf.~Remark~\ref{rem:alper}).

\section{Spherical roots adapted to $\wm$} \label{sec:spher-roots-adapt}

In this section $\wm$ denotes a normal, but \emph{not necessarily free},
submonoid of $\dw$. By combining results from \cite{luna-typeA,
  knop-lv, losev-uniqueness, bravi-pezzini-primwonderful-arxivv1} we will describe when a set of spherical
roots is `adapted' or `N-adapted' to $\wm$. In particular, in
Corollaries~\ref{cor:indiv_adapt-spher-roots} 
and~\ref{cor:indiv_Nadapt-spher-roots} we give an explicit
characterization for when an element $\sigma$ of the root lattice is
`adapted' or `N-adapted' to $\wm$. 

\begin{definition} \label{def:Nadapted}
We say that a subset $\Sigma$ of $\N\sr$ is
\textbf{N-adapted} to $\wm$ if there exists an affine spherical
$G$-variety $X$ such that $\wm(X) = \wm$ and $\Sigma^N(X) =
\Sigma$. By slight abuse of language, we say that an element $\sigma$ of $\N\sr$ is N-adapted to
$\wm$ if $\{\sigma\}$ is N-adapted to $\wm$. 
\end{definition}

We will give the definition of `adapted', which requires some more
notions from the theory of spherical varieties, in
Definition~\ref{def:adapted} below. After recalling some basic definitions
concerning spherical varieties, we briefly discuss, in
Section~\ref{sec:spherical-systems}, the notion of
`spherically closed spherical systems', and the role they play in
classifiying spherically closed spherical subgroups of $G$. We then, in Section~\ref{sec:augmentations} review
Luna's `augmentations'. They classify the subgroups of $G$ which have a
given spherical closure $K$. Finally, after recalling some basic
results from the Luna-Vust theory of spherical embeddings in
Section~\ref{sec:colored-cones}, we deduce the combinatorial
characterization of adapted and N-adapted spherical roots.  

\subsection{Basic definitions} \label{sec:basic-notions}
In this section we briefly recall the basic definitions of the theory
of spherical varieties
by freely quoting from \cite{luna-typeA}. 
For more details on these notions the reader can
also consult~\cite{pezzini-cirmspher, timashev-embbook}.

We recall that a (not necessarily
affine) $G$-variety $X$ is called \textbf{spherical} if it is normal
and contains an open dense orbit for $B$. If $X$ is affine, this is
equivalent to the definition given before in terms of $\k[X]$. 

The complement of the open $B$-orbit in $X$ consists of finitely many $B$-stable prime divisors. 
Among those, the ones that are \emph{not} $G$-stable are called the \textbf{colors} of $X$. 
The set of colors of $X$ is denoted by $\Delta_X$. 

By the \textbf{weight lattice} $\wl(X)$ of
$X$ we mean the subgroup of $\wl$ made up of the $B$-weights in the
field of rational functions $\k(X)$. 
Since $X$ has a dense $B$-orbit two rational $B$-eigenfunctions on $X$
of the same weight 
are scalar multiples of one another.
 
Let $P_X$ be the stabilizer of the open $B$-orbit and
denote by $S^p_X$ the subset of simple roots corresponding to
$P_X$, which is a parabolic subgroup of $G$ containing $B$. 

Let
$\mathcal V_X \inn \Hom(\wl(X),\Q)$ be
the so-called \textbf{valuation cone} of $X$, i.e.\ the set of $\Q$-valued $G$-invariant valuations on $\k(X)$ 
seen as functionals on $\wl(X)$. 
By \cite[Theorem~3.5]{brion-gensymm} $\mathcal V_X$ is a cosimplicial cone.
Let $\Sigma(X)$ be the set of linearly independent primitive elements in $\wl(X)$ such that 
\[\mathcal V_X=\{v\in\Hom(\wl(X),\Q):\langle v,\sigma\rangle\leq0\mbox{ for all }\sigma\in\Sigma(X)\},\] 
i.e.\ the set of \textbf{spherical roots} of $X$.

Similarly, the discrete valuations on $\k(X)$ associated with colors give rise to functionals on $\wl(X)$.
This yields the so-called \textbf{Cartan pairing} of $X$, a $\Z$-bilinear map denoted by
\[c_X\colon \Z\Delta_X\times\wl(X)\to\Z.\]

Since
$X$ has a dense $B$-orbit, it has a dense $G$-orbit. Let $H$ be the
stabilizer of a point in this orbit, which we can then identify with
$G/H$. The group $H$ is called a \textbf{spherical subgroup} of $G$
because $G/H$ is a spherical $G$-variety. To $H$, we can associate a larger group $\overline{H}$, called
the \textbf{spherical closure} of $H$: the normalizer of $H$ in $G$
acts by $G$-equivariant automorphisms on $G/H$ and $\overline{H}$ is
the kernel of the induced action of this normalizer on $\Delta_X$
(see \cite[\S 6.1]{luna-typeA} or \cite[\S 2.4.1]{bravi-luna-f4}). We
recall that it follows from \cite[Lemma~2.4.2]{bravi-luna-f4} that
$\overline{\overline{H}} = \overline{H}$ (see \cite[Proposition~3.1]{pezzini-redaut-arxivv2} for a direct proof).

\subsection{Spherical systems} \label{sec:spherical-systems}
Here we briefly recall the definition of spherical system and its
role in the classification of spherical varieties, 
see \cite{luna-typeA,   bravi-luna-f4}.  

Wonderful varieties are special spherical varieties satisfying certain regularity properties.
We do not need their definition here, 
we just recall that by~\cite{losev-uniqueness, bravi-pezzini-primwonderful-arxivv1} 
wonderful $G$-varieties (or their open $G$-orbits) 
are classified by
their so-called spherical systems. This was known as Luna's conjecture,
another proof of which was proposed in \cite{cupit-wvgr-prep}. By
\cite{knop-auto},  spherical homogeneous spaces
$G/K$ 
with $K$ \textbf{spherically closed} (that is, $\overline{K} = K$)  can be realized as the open $G$-orbit of a unique wonderful variety. Consequently, they correspond to spherically closed spherical $G$-systems 
(systems satisfying certain combinatorial conditions, as explained below):
\[G/K\longmapsto \mathscr S_{G/K}=(S^p_{G/K},\Sigma(G/K),\mathbf A_{G/K}).\]

Let $K$ be a spherically closed spherical subgroup of $G$. 
The set $\Sigma(G/K)$ of spherical roots of $G/K$ is included in the
root lattice $\rl$ 
(because $K$ contains the center of $G$) and it is a basis of
$\wl(G/K)$. Let $\mathbf A_{G/K}$ be the set of colors that are not stable
under some minimal parabolic containing $B$ and corresponding to a simple
root belonging to $\Sigma(G/K)$. The full Cartan pairing restricts
to the $\Z$-bilinear pairing
$c_{G/K}\colon \Z\mathbf A_{G/K}\times\Z\Sigma(G/K)\to\Z$,
also called restricted Cartan pairing.

\begin{definition} \label{def:sphericalroots} 
 The set $\Sigma^{sc}(G)$ of \textbf{spherically closed spherical roots of $G$} is defined as
\[\Sigma^{sc}(G):=\{\sigma \in \rl \colon \sigma \in \Sigma(G/K) \text{ for
  some spherically closed spherical subgroup $K$ of $G$}\}.\] 
Let $H$ be a spherical subgroup of $G$ and let $X$ be any spherical
$G$-variety with open $G$-orbit $G/H$. Let $\overline{H}$ be the
spherical closure of $H$. We define
\begin{align*}
\Sigma^{sc}(X)&:= \Sigma^{sc}(G/H) := \Sigma(G/\overline{H});\\
\Sigma^N(X)&:= \Sigma^{N}(G/H):= \Sigma(G/N_G(H)).
\end{align*}
\end{definition}

\begin{remark}
\begin{enumerate}[1.]
\item It follows from \cite[Theorem~1.3]{knop-auto} that for $X$ affine,
$\Sigma^N(X)$ given in Definition~\ref{def:sphericalroots} agrees with
the description in Section~\ref{sec:introduction} of the set of
N-spherical roots of $X$.  
\item Thanks to \cite[Theorem~2]{losev-uniqueness} one can precisely describe the
  relationship between the three sets $\Sigma(X)$, $\Sigma^{sc}(X)$
  and $\Sigma^N(X)$; see Proposition~\ref{prop:scversusN} and \cite{bvs-owr} for more information. 
\item While $\Sigma^{sc}(X)$ and $\Sigma^N(X)$ are subsets of $\N\sr$,
  there exist wonderful varieties $X$ such that $\Sigma(X) \not\subset
  \rl$ (see \cite{wasserman}).
\item $\Sigma(X)$ is not always a basis of $\wl(X)$, but it is
  when $X$ is wonderful. 
\item The weight lattice, valuation cone and spherical roots
  are birational invariants of the spherical variety $X$ since they only
  depend on its open $G$-orbit $G/H$. The same is true of the colors and
  the Cartan pairing once we (naturally) identify the colors of $G/H$ with
  their closures in $X$.  
\end{enumerate}
\end{remark}

The set $\Sigma^{sc}(G)$ is finite. More precisely, there is the next
proposition, which follows from the classification of spherically closed spherical subgroups $K$ of $G$ 
with $\wl(G/K)$ of rank 1 \cite{ahiezer-eqvcompl,losev-uniqueness}, see also \cite[\S~1.1.6 and \S~2.4.1]{bravi-luna-f4}. We recall
that the \textbf{support} $\supp(\sigma)$ of $\sigma \in \N\sr$ is the
set of simple roots which have a nonzero coefficient in the unique
expression of $\sigma$ as a linear combination of the simple roots.  

\begin{proposition} \label{prop:sigmaG} 
An element $\sigma$ of $\N\sr$ belongs to $\Sigma^{sc}(G)$ if and only if 
after numbering the simple roots in $\supp(\sigma)$ like Bourbaki (see \cite{bourbaki-geadl47}) 
$\sigma$ is listed in Table~\ref{table:spherical_roots}. 
\end{proposition}

\begin{table}\caption{spherically closed spherical roots} \label{table:spherical_roots}
\begin{center}
\begin{tabular}{ll}
Type of support & $\sigma$ \\
\hline
$\sf A_1$ & $\alpha$\\
$\sf A_1$ & $2\alpha$\\
$\mathsf A_1 \times \mathsf A_1$ & $\alpha+\alpha'$\\
$\mathsf A_n$, $n\geq 2$ & $\alpha_1+\ldots+\alpha_n$\\
$\mathsf A_3$ & $\alpha_1+2\alpha_2+\alpha_3$\\
$\mathsf B_n$, $n\geq 2$ & $\alpha_1+\ldots+\alpha_n$\\
                     & $2(\alpha_1+\ldots+\alpha_n)$\\
$\mathsf B_3$ & $\alpha_1+2\alpha_2+3\alpha_3$\\
$\mathsf C_n$, $n\geq 3$ & $\alpha_1+2(\alpha_2+\ldots+\alpha_{n-1})+\alpha_n$\\
%$\mathsf D_4$ & $\alpha_1+2\alpha_2+2\alpha_3+\alpha_4$ \\
%$\mathsf D_4$ & $\alpha_1+2\alpha_2+\alpha_3+2\alpha_4$ \\
$\mathsf D_n$, $n\geq 4$ & $2(\alpha_1+\ldots+\alpha_{n-2})+\alpha_{n-1}+\alpha_n$\\
$\mathsf F_4$ & $\alpha_1+2\alpha_2+3\alpha_3+2\alpha_4$\\
$\mathsf G_2$ & $4\alpha_1+2\alpha_2$\\          
%& $2\alpha_1+\alpha_2$ \\
& $\alpha_1+\alpha_2$
\end{tabular}
\end{center}
\end{table}

Recall that $K$ is a spherically closed spherical subgroup of $G$. 
Therefore, see \cite[\S 7.1]{luna-typeA}, the triple $\mathscr S_{G/K}=(S^p_{G/K},
\Sigma(G/K), \mathbf A_{G/K})$ is a spherically closed Luna spherical system in the following sense.
\begin{definition} \label{def:spherical-systems}
Let $(S^p,\Sigma,\mathbf A)$ be a triple where $S^p$ is a  subset of $\sr$,
$\Sigma$ is a subset of $\Sigma^{sc}(G)$ %without proportional elements
and $\mathbf A$ is a finite set endowed with a $\Z$-bilinear pairing $c\colon \Z\mathbf A\times\Z\Sigma\to\Z$. 
For every $\alpha \in \Sigma \cap S$, let $\mathbf A (\alpha)$ denote the set $\{D \in \mathbf A : c(D,\alpha)=1 \}$. 
Such a triple is called a \textbf{spherically closed spherical $G$-system} if all the
following axioms hold: 
\begin{itemize} 
\item[(A1)] for every $D \in \mathbf A$ and every $\sigma \in \Sigma$,
  we have that $c(D,\sigma)\leq 1$ and that
  if $c(D, \sigma)=1$ then $\sigma \in S$; 
\item[(A2)] for every $\alpha \in \Sigma \cap S$, $\mathbf A(\alpha)$
  contains two elements, which we denote by $D_\alpha^+$ and
  $D_\alpha^-$, and for all $\sigma \in \Sigma$ we have $c(D_\alpha^+,\sigma) + c(D_\alpha^-,\sigma) = \langle \alpha^\vee , \sigma \rangle$; 
\item[(A3)] the set $\mathbf A$ is the union of $\mathbf A(\alpha)$ for all $\alpha\in\Sigma \cap S$; 
\item[($\Sigma 1$)] if $2\alpha \in \Sigma \cap 2S$ then $\frac{1}{2}\langle\alpha^\vee, \sigma \rangle$ is a non-positive integer for all $\sigma \in \Sigma \setminus \{ 2\alpha \}$; 
\item[($\Sigma 2$)] if $\alpha, \beta \in S$ are orthogonal and $\alpha + \beta$ belongs to $\Sigma$ then $\langle \alpha ^\vee , \sigma \rangle = \langle \beta ^\vee , \sigma \rangle$ for all $\sigma \in \Sigma$; 
\item[(S)] every $\sigma \in \Sigma$ is \textbf{compatible} with
  $S^p$, that is, for every $\sigma \in \Sigma$ there exists a
  spherically closed spherical subgroup $K$ of $G$ with
  $S^p_{G/K}=S^p$ and $\Sigma(G/K)=\{\sigma\}$.
\end{itemize}
\end{definition}

\begin{remark}\label{rem:spherical-systems} 
\begin{enumerate}[1.]
\item Condition (S) of Definition~\ref{def:spherical-systems} can be stated
in purely combinatorial terms as follows (see \cite[\S 1.1.6]{bravi-luna-f4}).
A spherically closed spherical root $\sigma$ is compatible with $S^p$ if and only if:
\begin{itemize}
\item in case $\sigma=\alpha_1+\ldots+\alpha_n$ with support of type $\ssB_n$
\[\{\alpha\in\supp\sigma\colon\<\alpha^\vee,\sigma\>=0\}\setminus\{\alpha_n\}
\subseteq S^p\subseteq\{\alpha\in S\colon\<\alpha^\vee,\sigma\>=0\}\setminus\{\alpha_n\},\]
\item in case $\sigma=\alpha_1+2(\alpha_2+\ldots+\alpha_{n-1})+\alpha_n$ with support of type $\ssC_n$
\[\{\alpha\in\supp\sigma\colon\<\alpha^\vee,\sigma\>=0\}\setminus\{\alpha_1\}
\subseteq S^p\subseteq\{\alpha\in S\colon\<\alpha^\vee,\sigma\>=0\},\] 
\item in the other cases
\[\{\alpha\in\supp\sigma\colon\<\alpha^\vee,\sigma\>=0\}
\subseteq S^p\subseteq\{\alpha\in S\colon\<\alpha^\vee,\sigma\>=0\}.\]
\end{itemize} 
\item Definition~\ref{def:spherical-systems} combines the standard
  definition of spherical system, see \cite[\S 2]{luna-typeA}, with the
  requirement that it be spherically closed, see \cite[\S 7.1]{luna-typeA}
  and \cite[\S 2.4]{bravi-luna-f4}.
\end{enumerate}
\end{remark}

As shown in \cite{luna-typeA}, the set $\Delta_{G/K}$ of colors and the Cartan pairing $c$ of $G/K$ are uniquely determined by
$\mathscr S_{G/K}$, in the sense that they can be naturally identified with the set of
colors of and the full Cartan pairing
of $\mathscr S_{G/K}$, defined as follows.
Let $\mathscr S=(S^p,\Sigma,\mathbf A)$ be a (spherically closed) spherical $G$-system. The
\textbf{set of colors of $\mathscr S$} is the finite set $\Delta$
obtained as the disjoint union $\Delta=\Delta^a\cup\Delta^{2a}\cap\Delta^b$ where:
\begin{itemize}
\item $\Delta^a=\mathbf A$,
\item $\Delta^{2a}=\{D_\alpha : \alpha\in S\cap{\frac1 2}\Sigma\}$,
\item $\Delta^b=\{D_\alpha : \alpha\in S\setminus(S^p\cup\Sigma\cup{\frac1 2}\Sigma)\}/\sim$, where $D_\alpha\sim D_\beta$ if $\alpha$ and $\beta$ are orthogonal and $\alpha+\beta\in\Sigma$.
\end{itemize}

The \textbf{full Cartan pairing of $\mathscr S$} is the $\Z$-bilinear map $c\colon\Z\Delta\times\Z\Sigma\to\Z$ defined as:
\[c(D,\sigma)=\left\{\begin{array}{ll}
c(D,\sigma) & \mbox{ if $D\in\Delta^a$}; \\
{\frac1  2}\langle\alpha^\vee,\sigma\rangle & \mbox{ if $D=D_\alpha\in\Delta^{2a}$}; \\
\langle\alpha^\vee,\sigma\rangle & \mbox{ if $D=D_\alpha\in\Delta^b$}. 
\end{array}\right.\]

\subsection{Augmentations} \label{sec:augmentations} 
We continue to use $K$ for a spherically closed spherical subgroup of $G$. 
By \cite[Proposition~6.4]{luna-typeA} spherical homogeneous spaces
$G/H$ such that $\overline H$, the spherical closure of $H$, is equal
to $K$ are classified by their weight lattice, which is an
augmentation of $\mathscr S_{G/K}$ .
\begin{definition}\label{def:augmentation}
Let $\mathscr S=(S^p,\Sigma,\mathbf A)$ be a spherically closed
spherical $G$-system with Cartan pairing $c: \Z \mathbf{A} \times
\Z\Sigma \to \Z$. 
An \textbf{augmentation} of $\mathscr S$ is a
lattice $\Lambda'\subset\Lambda$ endowed with a pairing
$c'\colon\Z\mathbf A\times\Lambda'\to\Z$ such that
$\Lambda'\supset\Sigma$ and
\begin{itemize}
\item[(a1)] $c'$ extends $c$;
\item[(a2)] if $\alpha\in S\cap\Sigma$ then $c'(D_\alpha^+,\xi)+c'(D_\alpha^-,\xi)=\langle\alpha^\vee,\xi\rangle$ for all $\xi\in\Lambda'$;
\item[($\sigma1$)] if $2\alpha\in2S\cap\Sigma$ then $\alpha \notin
  \Lambda'$ and $\langle\alpha^\vee,\xi\rangle\in2\Z$ for all $\xi\in\Lambda'$;
\item[($\sigma2$)] if $\alpha$ and $\beta$ are orthogonal elements of
  $\sr$ with $\alpha+\beta\in\Sigma$ then
  $\langle\alpha^\vee,\xi\rangle=\langle\beta^\vee,\xi\rangle$ for all
  $\xi\in\Lambda'$; and
\item[(s)] if $\alpha\in S^p$ then $\langle\alpha^\vee,\xi\rangle=0$ for all $\xi\in\Lambda'$.
\end{itemize}
Let $\Delta$ be the set of colors of $\mathscr S$. The \textbf{full
  Cartan pairing} of the augmentation is the $\Z$-bilinear map $c': \Z \Delta
\times \Lambda'\to\Z$ given by
\begin{equation} \label{eq:4}
c'(D,\gamma)=\left\{\begin{array}{ll}
c'(D,\gamma) & \mbox{ if $D\in\Delta^a$}; \\
{\frac1  2}\langle\alpha^\vee,\gamma\rangle & \mbox{ if $D=D_\alpha\in\Delta^{2a}$}; \\
\langle\alpha^\vee,\gamma\rangle & \mbox{ if $D=D_\alpha\in\Delta^b$}. 
\end{array}\right.
\end{equation}
\end{definition}

\begin{remark} \label{rem:CPaugmentation}
By the definition of spherical closure, 
$\Delta_{G/H}$ and $\Delta_{G/\overline H}$ are naturally identified and the full Cartan pairing $\Z
\Delta_{G/H} \times \Lambda(G/H) \to \Z$ on
$G/H$ is the full Cartan pairing of the augmentation corresponding to
$H$ (see Proposition~6.4 and the proof of Theorem~3 in
\cite{luna-typeA}). 
\end{remark}

We state here, for future reference, the following consequence of
\cite[Theorem~2]{losev-uniqueness}. 
\begin{proposition} \label{prop:scversusN}
Let $G/H$ be a spherical homogeneous space with $\Sigma^{sc}(G/H) = \Sigma$. Then
\[\Sigma^N(G/H) = (\Sigma \setminus \Sigma_l) \cup 2\Sigma_l,\]
where $\Sigma_l=\{\alpha \in \Sigma \cap \sr \colon
c_{G/H}(D^+_{\alpha},\gamma) =  c_{G/H}(D^-_{\alpha},\gamma) \text{ for
  all }\gamma \in \wl(G/H)\}.$
\end{proposition}
\begin{proof}
This follows immediately from  comparing \cite[Theorem~2]{losev-uniqueness},
  which describes the relationship between $\Sigma(G/H)$ and
  $\Sigma^N(G/H)$ with \cite[Lemma~7.1]{luna-typeA}, which describes
  the relationship between $\Sigma(G/H)$ and $\Sigma^{sc}(G/H)$. Note
  that \cite[Lemma~7.1]{luna-typeA} can be deduced from
  \cite{losev-uniqueness} without appealing to Luna's conjecture. 
\end{proof}

\subsection{Strictly convex colored cones and weight monoids of affine spherical varieties} \label{sec:colored-cones}
An equivariant embedding of a spherical homogeneous space $G/H$ as a dense orbit in a spherical $G$-variety
(an embedding of $G/H$, for short) is called
\textbf{simple} if it has only one closed orbit. Affine spherical
varieties are simple. 

If $X$ is a simple embedding of the spherical homogeneous space $G/H$, 
let $\mathcal F(X)$ be the set of colors of $X$ containing the closed orbit
(identified with elements of $\Delta_{G/H}$),
and let $\mathcal C(X)$ be the cone in $\Hom(\wl(G/H), \Q)$ 
generated by the valuations associated with the $G$-stable divisors of $X$ 
(identified with elements of $\mathcal V_{G/H}$) and by $c(\mathcal F(X),\cdot)$.
The couple $(\mathcal C(X),\mathcal F(X))$ is a strictly convex colored cone in the sense of the following definition.

A \textbf{strictly convex colored cone} is a couple $(\mathcal C,\mathcal F)$ where 
\begin{itemize}
\item[-] $\mathcal F$ is a subset of $\Delta_{G/H}$ such that the subset 
$c(\mathcal F,\cdot)$ of $\Hom(\wl(G/H),\Q)$ does not contain $0$,
\item[-] $\mathcal C$ is a strictly convex polyhedral cone in $\Hom(\wl(G/H), \Q)$  which is generated by
$c(\mathcal F,\cdot)$ and finitely many elements of $\mathcal V_{G/H}$ and  
whose relative interior intersects $\mathcal V_{G/H}$. 
\end{itemize}
We recall from \cite[Theorem~3.1]{knop-lv} that simple embeddings $X$ of the spherical homogoneous space $G/H$ 
are classified by their strictly convex colored cones. By \cite[Theorem~6.7]{knop-lv} the simple embedding $X$ is affine if
and only if there exists a character $\chi\in\wl(G/H)$ that is non-positive on $\mathcal V_{G/H}$, zero
on $\mathcal C(X)$ and $c(\cdot,\chi)$ is strictly positive on
$\Delta_{G/H}\setminus\mathcal F(X)$. 

We gather some known results about the weight monoid of affine spherical
varieties.
\begin{proposition} \label{prop:weightmonoid}
If $X$ is an affine sperical $G$-variety with weight monoid
$\wm(X)$ and open orbit $G/H$, then
\begin{enumerate}[(a)]
\item the weight lattice of $X$ (or of $G/H$) is $\Z\wm(X)$; \label{item:5}
\item the set $S^p_X$ (which is the same as $S^p_{G/H}$) is equal to
  $\{\alpha \in \sr \colon \<\alpha^{\vee},\gamma\>=0\mbox{ for all }\gamma\in\wm(X)\}$; \label{item:8}
\item the dual cone \(\wm^{\vee}(X):=\{v\in\Hom(\Z\wm(X),\Q):
  \langle v,\gamma\rangle\geq0\mbox{ for all }\gamma\in\wm(X)\}\) to
  $\wm(X)$ is a strictly convex polyhedral cone; \label{item:11}
\item every ray of \(\wm^{\vee}(X)\) contains an element of
  $c(\Delta_{G/H},\cdot)$ or of $\mathcal
  V_{G/H}$; \label{item:9}
\item \(\wm^{\vee}(X)\) contains $c(\Delta_{G/H},\cdot)$. \label{item:10}
\end{enumerate} 
\end{proposition}
\begin{proof}
These statements are well-known to experts and can be extracted from the results summarized in \cite[\S 15.1]{timashev-embbook}. For the reader's convenience, we provide a proof. 
Assertion (\ref{item:5}) follows from the fact that a rational $B$-eigenfunction on $X$ 
is necessarily equal to the quotient of two regular $B$-eigenfunctions; see e.g.~\cite[Proposition
2.8(i)]{brion-cirmactions}.
Assertion (\ref{item:8}) is \cite[Lemme 10.2]{camus}. It follows from the fact that $P_X$ is the common stabilizer of the $B$-stable lines in $\k[X]$.  
This is the case because $P_X$ is the common stabilizer of the $B$-stable prime divisors of $X$ and the union of these divisors is the zero set of some $B$-eigenvector in $\k[X]$.  
Assertion (\ref{item:11}) is a standard fact in convex
geometry.  Parts (\ref{item:9}) and (\ref{item:10}) follow from the
fact that a rational $B$-eigenfunction on $X$ is regular if and only if it does not have poles along the colors or $G$-stable prime divisors of $X$. This, in turn, is so because $X$ is normal.
\end{proof}

\subsection{Adapted spherical roots} \label{subsec:adapt-spher-roots}
Recall that $\wm$ is a normal
submonoid of $\dw$. Combining the results recalled above, one derives
the condition on a set of spherical roots $\Sigma$ for being adapted
to $\wm$. 

\begin{definition} \label{def:adapted}
We say that a subset $\Sigma$ of $\Sigma^{sc}(G)$ is \textbf{adapted} (or \textbf{N-adapted}) to
$\wm$ if there exists an affine spherical $G$-variety $X$ such that
$\wm(X) = \wm$ and 
$\Sigma^{sc}(X) = \Sigma$ (respectively, $\Sigma^N(X) = \Sigma$). 
\end{definition}

\begin{remark} \label{rem:unique_adapt_var}
Let $\Sigma$ be a subset of $\Sigma^{sc}(G)$. 
Losev's Theorem~\cite[Theorem~1.2]{losev-knopconj} asserts that there
is \emph{at most one} affine spherical $G$-variety $X$ with $\wm(X) =
\wm$ and $\Sigma^N(X)=\Sigma$. Because $\Sigma^{sc}(X)$ determines
$\Sigma^N(X)$  (see 
Proposition~\ref{prop:scversusN}) there is also at most one
affine spherical $G$-variety $Y$ with $\Sigma^{sc}(Y)= \Sigma$ and
$\wm(Y) = \wm$. 
\end{remark}

The dual cone to $\wm$ is
\[\wm^{\vee}:=\{v\in\Hom(\Z\wm,\Q):
  \langle v,\gamma\rangle\geq0\mbox{ for all }\gamma\in\wm\}.\]
It is a strictly convex polyhedral cone. We denote the set of primitive vectors
on its rays by $E(\wm)$:
\begin{equation}
E(\wm):= \{\delta \in (\Z\wm)^* \colon \delta
\text{ spans a ray of }\wm^\vee \text{ and } \delta \text{ is primitive}\}.
\end{equation}
Observe that
\begin{equation} \label{eq:1}
E(\wm)=\{\delta \in (\Z\wm)^*\colon \delta  \text{
    is primitive},
\delta(\wm) \inn \Z_{\ge 0}, 
\delta \text{ is the equation of a face
  of codim $1$ of $\Q_{\ge 0}\wm$}\}. \end{equation} 
Moreover, for 
$\alpha \in S \cap \Z\wm$, we define
\[a(\alpha):= \{\delta \in  (\Z\wm)^*\colon \delta(\alpha)=1 \text{
  and } \bigl(\delta
\in E(\wm) \text{ or } \alpha^{\vee} - \delta \in E(\wm)\bigr)\}.\]
Finally, we put
\[S^p(\wm):=\{\alpha\in
  \sr:\langle\alpha^\vee,\gamma\rangle=0 \text{ for all
  }\gamma\in\wm\}.\]

\begin{proposition} \label{prop:adapt-spher-roots}
Let $\wm$ be a normal submonoid
of $\dw$. A subset $\Sigma$ of $\Sigma^{sc}(G)$ is adapted to $\wm$ if and
only if there exists a spherically closed spherical system $\mathscr
  S=(S^p, \Sigma, \mathbf A)$ such that
\begin{enumerate}
\item  $S^p=S^p(\wm)$; and \label{item:1}
\item  $\Z\wm$ is an augmentation of $\Z\Sigma$; and \label{item:2}
\item if $\delta \in E(\wm)$, then $\<\delta, \sigma\> \leq 0$ for all $\sigma \in
\Sigma$ or there exists $D\in \Delta$ such that $c(D, \cdot)$ is a
positive multiple of $\delta$; where $\Delta$ is the set of colors of
$\mathscr S$ and $c: \Z\Delta \times \Z\wm \to \Z$ is the full Cartan
pairing of the augmentation; and \label{item:3}
\item $ c(D,\cdot) \in \wm^\vee$ for all $D \in \Delta$. \label{item:4}
\end{enumerate}
\end{proposition}

\begin{proof}
This is a consequence of the results we reviewed in
Sections~\ref{sec:spherical-systems}
through~\ref{sec:colored-cones}. We begin with the \emph{necessity} of the
conditions. Let $X$ be an affine spherical $G$-variety with
$\Sigma^{sc}(X) = \Sigma$ and $\wm(X) = \wm$. Let $G/H$ be the open
orbit of $X$ and let $\overline{H}$ be the spherical closure of
$H$. Then $\Sigma^{sc}(X) = \Sigma(G/\overline{H})$ by definition, and
$S^p_{G/H} = S^{p}(\wm)$ by
Proposition~\ref{prop:weightmonoid}(\ref{item:8}). Moreover $S^p_{G/H}
= S^p_{G/\overline{H}}$.  
It follows from  \S 5.1 and Lemma~7.1 in \cite{luna-typeA} that 
$(S^p(\wm), \Sigma, \mathbf A_{G/\overline{H}})$ is a spherically
closed spherical system. 
Since $H$ has spherical closure $\overline{H}$, (\ref{item:2}) follows
from \cite[Proposition~6.4]{luna-typeA}. Conditions (\ref{item:3}) and
(\ref{item:4}) follow from (\ref{item:9}) and (\ref{item:10}) of
Proposition~\ref{prop:weightmonoid}. 

We now show that the conditions are \emph{sufficient} for $\Sigma$ to be
adapted to $\wm$. By \cite{bravi-pezzini-primwonderful-arxivv1} there
exists a spherically closed spherical subgroup $K$ of $G$ with spherical system
$\mathscr S$. Condition (\ref{item:2}) implies by \cite[Proposition
6.4]{luna-typeA} that there exists a spherical subgroup $H$ of $G$
with $\overline{H} = K$ and $\wl(G/H) = \Z\wm$. What remains is to
prove that $G/H$ has an affine embedding $X$ with weight monoid
$\wm$. That is, by \cite[Theorems 3.1 and 6.7]{knop-lv} we have to show that 
there exists a strictly convex colored cone $(\mathcal C,\mathcal F)$ in
  $\Hom(\Z\wm,\Q)$, with respect to $\mathcal
  V=\{v\in\Hom(\Z\Gamma,\Q): \langle
  v,\sigma\rangle\leq0\mbox{ for all }\sigma\in\Sigma\}$ and the set
  of colors $\Delta$ of $\mathscr S$, such that:
\begin{enumerate}[(i)]
\item there exists $\chi\in\Z\Gamma$ that is non-positive on
  $\mathcal V$, zero on $\mathcal C$ and strictly positive on
  $\Delta\setminus\mathcal F$; and
\item $\wm=\{\gamma\in\Z\Gamma:\langle
  v,\gamma\rangle\geq0\mbox{ for all }v\in\mathcal C\cup\Delta\}$. \label{item:12}
\end{enumerate}
We claim that if (\ref{item:1}), (\ref{item:3}) and (\ref{item:4}) hold, then the
desired strictly convex colored cone exists. Indeed, take $\mathcal C$ to be the
maximal face of $\wm^\vee$ whose relative interior meets $\mathcal V$
with $\mathcal F$ the set of colors contained in $\mathcal C$ (such a
maximal face exists since the zero face actually meets $\mathcal V$). 
The set $c(\mathcal F,\cdot)$ does not contain $0$.
Indeed, a color $D$ with $c(D,\cdot)=0$ necessarily belongs to $\Delta^b$, whence
$c(D,\cdot)=\langle\alpha^\vee,\cdot\rangle$ for some $\alpha\in \sr$
but by (\ref{item:1}) this implies $\alpha\in S^p$.
Moreover, $\mathcal C$ is contained in a hyperplane that separates
$\mathcal V$ and $\Delta\setminus\mathcal F$.  This yields $\chi$. 
The inclusion ``$\subset$''
of (\ref{item:12}) holds because $\mathcal C \inn
\wm^\vee$ and because $c(\Delta,\cdot) \inn \wm^\vee$ by
(\ref{item:4}). The other inclusion follows from (\ref{item:3}) and the
maximality of $\mathcal C$. 
\end{proof}

\begin{remark} \label{rem:adapt_unique}
It follows from equation (\ref{eq:5}) below that the spherical
system $\mathscr S$ and the Cartan pairing of the augmentation in
Proposition~\ref{prop:adapt-spher-roots} are uniquely determined by
$\wm$ and $\Sigma$. 
\end{remark}

\begin{corollary} \label{cor:Nadapt-spher-roots}
Let $\wm$ be a normal submonoid
of $\dw$. A subset $\Sigma$ of $\Sigma^{sc}(G)$ is N-adapted to $\wm$ if and
only if there exists a subset $\widetilde{\Sigma}$ of $\Sigma^{sc}(G)$ which is
adapted to $\wm$ and such that $\Sigma = (\widetilde{\Sigma}\setminus
\widetilde{\Sigma}_l) \cup 2\widetilde{\Sigma}_l$, where 
$\widetilde{\Sigma}_l=\{\alpha \in \widetilde{\Sigma} \cap S \colon
  a(\alpha) \text{ has one element}\}.$  
\end{corollary}
\begin{proof}
This is a consequence of
Proposition~\ref{prop:adapt-spher-roots} and
Proposition~\ref{prop:scversusN} once we show the following: if $c$
is the full Cartan pairing of an augmentation $\Z\widetilde{\Sigma}
\inn \Z\wm$ of a spherical system $\mathscr S = (S^p(\wm),
\widetilde{\Sigma}, \mathbf{A})$ as in
Proposition~\ref{prop:adapt-spher-roots}, 
then 
\begin{equation} \label{eq:5}
a(\alpha) = \{c(D^+_{\alpha}, \cdot), c(D^-_{\alpha}, \cdot)\}
\end{equation}
for all $\alpha \in \widetilde{\Sigma} \cap \sr$. To prove the
inclusion ``$\subset$'' in (\ref{eq:5}),  let $\delta \in
a(\alpha)$. Then, $\<\delta, \alpha\> = \<\alpha^{\vee} -
\delta, \alpha\> = 1 > 0$ and at least one of $\delta$ and $\alpha^{\vee} - \delta$ is in
$E(\wm)$. By (\ref{item:3}) in
Proposition~\ref{prop:adapt-spher-roots} it follows that 
$\{\delta,\alpha^{\vee} - \delta\}$ contains a positive rational
multiple of $c(D,\cdot)$ for some color $D$. By axiom (A1) of the
spherical system $\mathscr S$, and the description (\ref{eq:4}) of
$c$, the color $D$ must be $D^+_{\alpha}$ or $D^-_{\alpha}$. Since
$c(D^+_{\alpha}, \alpha)=c(D^-_{\alpha}, \alpha) =1$, this implies
that 
the two sets $\{\delta,\alpha^{\vee} - \delta\}$ and $\{c(D^+_{\alpha},
\cdot),c(D^-_{\alpha},\cdot) \}$ intersect, and so by axiom (a2) of the
augmentation, they are equal. For the reverse inclusion in
(\ref{eq:5}) we have to show that $c(D^+_{\alpha}, \cdot)$ or
$c(D^-_{\alpha}, \cdot)$ belongs to $E(\wm)$. 
If neither belongs to $E(\wm)$, then by 
(\ref{item:3}) and (\ref{item:4}) in Proposition~\ref{prop:adapt-spher-roots} 
together with the description (\ref{eq:4}) of $c$
and axiom (A1) in Definition~\ref{def:spherical-systems}, 
each of them is a linear combination with positive rational coefficients 
of elements of $\Hom(\Z\wm, \Q)$ which are nonpositive on $\alpha$. 
This contradicts the fact that
$c(D^+_{\alpha}, \alpha) =1$ and finishes the proof of equation (\ref{eq:5}). 
\end{proof}

As the next two corollaries show, one can characterize very explicitly
 whether a single spherical root is adapted
(Corollary~\ref{cor:indiv_adapt-spher-roots}) or N-adapted
(Corollary~\ref{cor:indiv_Nadapt-spher-roots}) to $\wm$. In a 2005
working document, Luna had proposed a statement like Corollary~\ref{cor:indiv_adapt-spher-roots}. We remark
that while Proposition~\ref{prop:adapt-spher-roots} and
Corollary~\ref{cor:Nadapt-spher-roots} depend on the full
classification of wonderful varieties by spherical systems (Luna's
conjecture), the next two results only use the combinatorial
classification of rank $1$ wonderful varieties, which was obtained in
\cite{brion-rank1} and also in
\cite{ahiezer-eqvcompl}. 

\begin{corollary} \label{cor:indiv_adapt-spher-roots}
Let $\wm$ be a normal submonoid of $\dw$. 
If $\sigma \in \Sigma^{sc}(G)$, then $\sigma$ is adapted to $\wm$
if and only if all of the following conditions hold:
\begin{enumerate}[(1)]
\item $\sigma \in \Z\wm $; \label{sphad1}
\item $\sigma$ is compatible with $S^p(\wm)$; \label{sphad2}
\item if $\sigma \notin \sr$ and $\delta \in E(\wm)$ such that
  $\<\delta, \sigma\> > 0$ then there exists $\beta \in S\setminus
  S^p(\wm)$ such that $\beta^\vee$ is a positive multiple of
  $\delta$;  \label{sphad3}
\item if $\sigma \in \sr$ then \label{sradp}
\begin{enumerate}[(a)]
\item $a(\sigma)$ has one or two elements; and \label{sradp1}
\item $\<\delta, \gamma\> \ge 0$ for all $\delta \in a(\sigma)$ and
  all $\gamma \in \wm$; and \label{sradp2}
\item $\<\delta, \sigma\> \le 1$ for all $\delta \in E(\wm)$; \label{sradp3}
\end{enumerate}
\item if $\sigma = 2\alpha \in 2\sr$, then $\alpha \notin \Z\wm$ and $\<\alpha^{\vee}, \gamma\>
  \in 2\Z$ for all $\gamma \in \wm$; \label{sphad5}
\item if $\sigma = \alpha + \beta$ with $\alpha,\beta \in \sr$
  and \label{sorths} \label{sphad6}
  $\alpha \perp \beta$, then $\alpha^{\vee} = \beta^{\vee}$ on $\wm$. 
%\item if $\supp(\sigma)$ is of type $\ssB_k$ and $\sigma =
%  \alpha_1 + \alpha_2 +\ldots+ \alpha_k$, then $\alpha_k \notin
%  S^p(\wm)$. \label{sphad7}
\end{enumerate}
\end{corollary}
\begin{proof}
Let $\sigma \in \Sigma^{sc}(G)$. Define the triple $\mathscr S$ by
\[\mathscr S:= \begin{cases} (S^p(\wm), \{\sigma\}, \emptyset)
  &\text{if $\sigma \notin \sr$}; \\
(S^p(\wm), \{\sigma\}, \{D^+_{\sigma}, D^-_{\sigma}\}) &\text{if $\sigma \in \sr$}.
\end{cases} \]
Let $\Delta$ be the set of colors of $\mathscr S$ (see
Section~\ref{sec:spherical-systems}) and let $c:\Z\Delta \times \Z\wm$
be the bilinear pairing given by equation (\ref{eq:4}) if $\sigma
\notin \sr$ and by
\begin{align} \label{eq:6}
&c(D, \gamma)= \<\alpha^{\vee}, \gamma\> \text{ if $D = D_{\alpha} \in
  \Delta^b$}; \\
&\{c(D^+_{\sigma}, \cdot), c(D^-_{\sigma}, \cdot)\} = a(\sigma),
\end{align}
if $\sigma \in \sr$. By Remark~\ref{rem:adapt_unique}, we have to show
that the conditions of the corollary hold if and only if $\mathscr S$
is a spherically closed spherical system of which $\Z\wm$ together with
$c$ is an augmentation such that conditions (\ref{item:3}) and (\ref{item:4})
of Proposition~\ref{prop:adapt-spher-roots} hold. We briefly describe
the straightforward verification. 

We begin with the case $\sigma \notin \sr$. Then we have that
$\mathscr S$ is a spherically closed spherical $G$-system if and only if (\ref{sphad2})
holds. Then $c$ gives an augmentation of
$\mathscr S$ if and only if  (\ref{sphad1}), (\ref{sphad5}) and
(\ref{sphad6}) hold.  Condition  (\ref{item:4}) of Proposition~\ref{prop:adapt-spher-roots} 
is vacuous since $\wm \inn \Lambda^+$ and every $c(D,\cdot)$ is a
positive multiple of a coroot. Condition (\ref{sphad3}) in the
corollary is the same as condition (\ref{item:3}) of
Proposition~\ref{prop:adapt-spher-roots} by the definition of $c$.

We proceed to the case $\sigma \in \sr$. Now $\mathscr S$
is a spherically closed spherical $G$-system if and only if (\ref{sphad2}) and
(\ref{sradp1}) hold. Next, by construction, $c$ gives an augmentation of
$\mathscr S$ if and only if we have (\ref{sphad1}). Condition
(\ref{item:4}) of Proposition~\ref{prop:adapt-spher-roots} is
equivalent to (\ref{sradp2}). Finally, condition
(\ref{item:3}) of Proposition~\ref{prop:adapt-spher-roots} is
equivalent to (\ref{sradp3}), again by the definition of $c$.  
\end{proof}

The combinatorial conditions that characterize N-adapted spherical roots 
are exactly the same except for conditions (\ref{sradp1}) and (\ref{sphad5}).
We report all of them again entirely in the next statement for later reference.

\begin{corollary} \label{cor:indiv_Nadapt-spher-roots}
Let $\wm$ be a normal submonoid of $\dw$. 
If $\sigma \in \Sigma^{sc}(G)$, then $\sigma$ is N-adapted to $\wm$
if and only if all of the following conditions hold:
\begin{enumerate}[(1)]
\item $\sigma \in \Z\wm $; \label{item:inasr1}
\item $\sigma$ is compatible with $S^p(\wm)$; \label{item:inasr2}
\item if $\sigma \notin \sr$ and $\delta \in E(\wm)$ such that
  $\<\delta, \sigma\> > 0$ then there exists $\beta \in S\setminus
  S^p(\wm)$ such that $\beta^\vee$ is a positive multiple of
  $\delta$; \label{item:inasr3}

\item if $\sigma \in \sr$ then   \label{item:inasr4}
\begin{enumerate}[(a)]
\item $a(\sigma)$ has two elements; and    \label{item:inasr4a}
\item $\<\delta, \gamma\> \ge 0$ for all $\delta \in a(\sigma)$ and
  all $\gamma \in \wm$; and     \label{item:inasr4b}
\item $\<\delta, \sigma\> \le 1$ for all $\delta \in E(\wm)$; \label{item:inasr4c}
\end{enumerate}
\item if $\sigma = 2\alpha \in 2\sr$, then $\<\alpha^{\vee}, \gamma\>
\in 2\Z$ for all $\gamma \in \wm$; \label{item:inasr5}
\item if $\sigma = \alpha + \beta$ with $\alpha,\beta \in \sr$ and \label{nsorths}
  $\alpha \perp \beta$, then $\alpha^{\vee} = \beta^{\vee}$ on $\wm$. 
%\item if $\supp(\sigma)$ is of type $\ssB_k$ and $\sigma =
%  \alpha_1 + \alpha_2 +\ldots +\alpha_k$, then $\alpha_k \notin S^p(\wm)$.
\end{enumerate}
\end{corollary}
\begin{proof}
By Corollary~\ref{cor:Nadapt-spher-roots}, if $\sigma \notin \sr \cup
2\sr$, then $\sigma$ is adapted to $\wm$ if and only if it is
N-adapted to $\wm$. From the same corollary it follows that $\sigma
\in \sr$ is N-adapted to $\wm$ if and only if it is adapted to $\wm$
and $a(\sigma)$ has two elements. 
The only remaining case is $\sigma = 2\alpha$ for some $\alpha \in
\sr$. Again by Corollary~\ref{cor:Nadapt-spher-roots}, $2\alpha$ is
N-adapted to $\wm$ if and only if either
\begin{enumerate}[(i)]
\item $2 \alpha$ is adapted to $\wm$; or \label{item:13}
\item $\alpha$ is adapted to $\wm$ and $a(\alpha)$ has one element. \label{item:14}
\end{enumerate}
We assume that (\ref{item:inasr1}) and (\ref{item:inasr2})
hold and claim that (\ref{item:inasr3}) and (\ref{item:inasr5}) hold
if and only if (\ref{item:13}) or (\ref{item:14}) is true. Indeed, it is clear
from Corollary~\ref{cor:indiv_adapt-spher-roots} that if $2\alpha$ is
adapted to $\wm$ then we have (\ref{item:inasr3}) and
(\ref{item:inasr5}). On the other hand, if $\alpha$ is adapted to
$\wm$ and $a(\alpha)$ has one element, then that element is 
$\frac{1}{2}\alpha^{\vee}$ and so (\ref{item:inasr5}) holds. Moreover,
condition (\ref{sradp3}) of
Corollary~\ref{cor:indiv_adapt-spher-roots} implies
(\ref{item:inasr3}) of this corollary. Conversely, suppose that we
have (\ref{item:inasr3}) and (\ref{item:inasr5}). Since the restricion of
$\alpha^{\vee}$ to $\Z\wm$ belongs to $\wm^\vee$ and $\<\alpha^{\vee},
2\alpha\> > 0$, there exists $\delta \in E(\wm)$ such that $\<\delta,
2\alpha\>>0$. It follows from (\ref{item:inasr3}) that $\delta =
q\beta^{\vee}$ for some $\beta \in \sr \setminus S^p(\wm)$ and $q \in
\Q_{>0}$. Clearly, $\beta = \alpha$, which proves that $\delta$ is the
only element of $E(\wm)$ that takes a positive value on
$2\alpha$. Now, suppose that $2\alpha$ is not
adapted to $\wm$, i.e. that (\ref{item:13}) does not hold. Then $\alpha$ must be an element of  $\Z\wm$. By (\ref{item:inasr5}),
$\frac{1}{2}\alpha^{\vee}$ takes integer values on $\Z\wm$, and since
it takes value $1$ on $\alpha$, it is primitive in $(\Z\wm)^*$ and
therefore an element of $E(\wm)$ and the only element of
$a(\alpha)$. It follows from
Corollary~\ref{cor:indiv_adapt-spher-roots} that (\ref{item:14})
is true. This finishes the proof. 
\end{proof}

\section{The $\Tad$-weights in $\Vgg$} \label{sec:tad-weights-vgg}
For the remainder of the paper, $\wm$ will be a free monoid with basis
$F \inn \dw$. In this section, we begin by recalling that the moduli
scheme $\ms$ is an open subscheme of a certain invariant Hilbert
scheme $\hs$. This allows one to realize the tangent space $\tg \ms$ as a $T$-submodule of a certain vector space $\Vgg$.  
In Section~\ref{sec:vggweightsr} we prove that if $\gamma$ is a
$T$-weight in $\Vgg$, then it is a spherical root of spherically closed type. 
In Section~\ref{sec:vggweightsprops} we further show that $\gamma$ is
compatible with $S^p(\wm)$. We also show
that if $\gamma \notin \sr$, then the weight space $\Vgg_{(\gamma)}$
has dimension at most $1$. For notational and computational
convenience, we actually work with the opposite of Alexeev and Brion's
$T$-action on $\ms$ and with a twist of their action on $\hs$ (see Section~\ref{sec:ihsrev}). 

\subsection{The invariant Hilbert scheme and its tangent space} \label{sec:ihsrev}
We briefly review some known facts regarding $\ms$ and its relation to
a certain invariant Hilbert scheme $\hs$.  For more details we refer
to \cite{alexeev&brion-modaff}, \cite[Section 4.3]{brion-ihs} and to \cite[\S 2.1 and \S 2.2]{degensphermod}.  Recall that $\wm$ is a free monoid of dominant weights with basis $F=\{\lambda_1, \lambda_2, \ldots, \lambda_r\}$, and put
\begin{align*}
V &:=  V(\lambda_1)\oplus V(\lambda_2) \oplus \ldots \oplus V(\lambda_r); \\
\hwv &:= v_{\lambda_1} + v_{\lambda_2}+\ldots +v_{\lambda_r}. 
\end{align*}
 We denote by $\hs$ the Hilbert scheme $\tHilb^G_h(V)$ of
\cite{alexeev&brion-modaff}, where $h$ is the characteristic function
of $\wm^*:=-w_0\wm$ (where $w_0$ is the longest element in the
Weyl group of $G$). The scheme $\hs$ parametrizes the $G$-stable ideals $I$ of $\k[V]$ such that 
$\k[V]/I \isom \oplus _{\lambda \in \wm^*} V(\lambda)$ as $G$-modules. We equip $\hs$ with the action of $T$ described in
\cite[\S 2.2]{degensphermod}. This is the same action as in
\cite{bravi&cupit}, and is a `twist' of the action in \cite{alexeev&brion-modaff} and in \cite[p. 101]{brion-ihs}. 
We briefly recall its definition. Let $\GL(V)^G$ be the group of linear automorphisms of $V$ that commute with the action of $G$. Note that $\GL(V)^G$ is a torus of dimension $r$. The natural action of $\GL(V)^G$ on $V$ (by $G$-equivariant automorphisms) induces an action on $\hs$. Composing with the homomorphism 
\begin{equation}
T \to \GL(V)^G \colon t \mapsto (\lambda_1(t), \lambda_2(t), \ldots, \lambda_r(t)),
\end{equation}
we obtain our action of $T$ on $\hs$. 

The center $Z(G)$ of $G$ belongs to the kernel of this action, which therefore descends to an action of $\Tad:=T/Z(G)$. 
We will refer to our action as the ``$\Tad$-action'' on $\hs$.
For the reader's convenience, the corresponding $\Tad$-action induced on the tangent space to $\hs$ at the $\Tad$-fixed point
is recalled below in (\ref{eq:tadaction}).  
As was reviewed in \cite[\S 2.2]{degensphermod} it follows from
\cite[Corollary~1.17 and Lemma~2.2]{alexeev&brion-modaff} that since
$\wm^*$ is free, we can view $\tM_{\wm^*}$ as a $\Tad$-stable open
subscheme of $\hs.$ Under this identification, the $\Tad$-fixed point $X_0$ of $\tM_{\wm^*}$ corresponds to a certain subvariety of $V$ which we also denote by $X_0$, namely
\begin{equation} \label{eq:7}
X_0 = \text{ the closure of the $G$-orbit of $\hwv$ in $V$}.
\end{equation}

The next proposition relates $\ms$ to $\hs$. 
\begin{proposition} \label{prop:msintohs}
Let $\wm$ be a free monoid of dominant weights. If we equip $\ms$ with
the opposite of the $\Tad$-action in \cite{alexeev&brion-modaff} and $\hs$
with the $\Tad$-action in \cite[\S 2.2]{degensphermod}, then there is a
$\Tad$-equivariant open embedding
\[\ms \into \hs\]
which sends the unique $\Tad$-fixed point of $\ms$ to the point $X_0$
in equation (\ref{eq:7}).
\end{proposition}
\begin{proof}
This a matter of ``formal bookkeeping.'' Composing the action of $G$
on $V(\wm)$ with the Chevalley involution
of $G$ induces an isomorphism $\ms \isom
\tM_{\wm^*}$. Composing this isomorphism with the open $\Tad$-equivariant embedding
$\tM_{\wm^*} \into \hs$ chosen above gives an open embedding $\ms
\into \hs$. Comparing the definition of the action in
\cite{alexeev&brion-modaff} with that of the action in \cite[\S
2.2]{degensphermod} one shows that this open embedding is
$\Tad$-equivariant for the actions as given in the proposition. 
\end{proof}

\begin{remark}
In what follows, $\ms$ and $\hs$ will always be equipped with the action given in Proposition~\ref{prop:msintohs}. The action Alexeev and Brion defined on $\ms$ is conceptually the most
natural, while we find the action we are using on $\hs$ computationally more convenient.
\end{remark}

By \cite[Proposition~1.13]{alexeev&brion-modaff}, there is a canonical isomorphism \[\tg \hs\isom H^0(X_0, \shN_{X_0|V})^G\] where $H^0(X_0, \shN_{X_0|V})^G$ is the space  of $G$-invariant global sections of the normal sheaf $\shN_{X_0|V}$ of $X_0$ in $V$. Moreover, by \cite[Proposition~3.10]{brion-ihs}, there is an inclusion of $\Tad$-modules 
\[H^0(X_0, \shN_{X_0|V})^G \into \Vgg \isom H^0(G\cdot x_0, \shN_{X_0|V})^G,\] where the $\Tad$-action on $\Vgg$ is induced by the following action of $\Tad$ on $V$. 
For $t \in \Tad$ and $v$ a $T$-weight vector of weight $\delta$ in  $V(\lambda) \inn V$, we put 
\begin{equation}\label{eq:tadaction}
t \cdot v := \lambda(t)\delta(t)^{-1}v.
\end{equation}  

\subsection{The $\Tad$-weights in $\Vgg$ are spherical roots of $G$} \label{sec:vggweightsr}
In this section, we prove the following theorem.

\begin{theorem}\label{thm:vggweights}
If $\gamma$ is a $\Tad$-weight in $\Vgg$ then $\gamma$ is a 
spherically closed spherical root of $G$. 
\end{theorem}
\begin{proof}
Corollary~\ref{cor:Vggweightposroot} and Corollary~\ref{cor:Vggweightnotposroot}.
\end{proof}

For future use, we recall the following elementary and well-known facts
regarding $\Vgg$. We include  proofs for convenience. Before stating them 
we define
\[F^{\perp}:= \{\beta \in \pr \colon \<\lambda, \beta^\vee\> = 0
\text{ for all }\lambda \in F\}.\]

\begin{proposition} \label{prop:elemVggfacts}
\begin{enumerate}[(a)]
\item A basis of $\Tad$-eigenvectors of $\fg \cdot x_0$ is given by
  $\{v_{\lambda} \colon \lambda \in F\} \cup \{X_{-\beta}\cdot x_0\colon \beta
    \in \pr \setminus F^{\perp}\}$. \label{item:6}
\item If $[v]$ is a $\Tad$-eigenvector in $\Vg$ of weight $\gamma$,
  then $[v] \in \Vgg$ if and only if $\gamma \in \Z\wm$ and $X_{\beta}
  \cdot v \in \fg \cdot x_0$ for all $\beta \in \sr \cup -(\sr \cap
  F^{\perp})$. \label{item:7}
\end{enumerate}
\end{proposition}

\begin{proof}
Assertion (\ref{item:6}) follows from the fact that $\fg \cdot x_0 =
\fb^{-} \cdot x_0 = \ft \cdot x_0 + \fn^{-}\cdot x_0$ and that $F$
is linearly independent. Assertion (\ref{item:7}) follows from
\cite[Lemma~2.16]{degensphermod} and the fact that $\fg_{x_0}$ is
generated as a Lie algebra by $\ft_{x_0}$ and the root spaces $\fg_\beta$ with $\beta \in \sr \cup -(\sr \cap
  F^{\perp})$ (see, e.g., \cite[Theorem~30.1]{humphreys-lag}). 
\end{proof}

\emph{In the remainder of this section,  $\gamma$ is a $\Tad$-weight occuring in $\Vgg$
and  $v \in V$  a $\Tad$-eigenvector of weight $\gamma$ such that
$[v]$ is a nonzero element of $\Vgg$.}
By Propostion~\ref{prop:elemVggfacts} (and the choice of our $\Tad$-action), the weight $\gamma$ belongs to $\N\sr\cap \Z\wm$.

\begin{lemma}[{\cite[Lemma~3.3]{bravi&cupit}}] \label{lem:simpluspos}\
\begin{enumerate}
\item There exists at least one simple root $\alpha$ such that
  $X_\alpha  v\neq0$. \label{item:16}
\item If $\alpha$ is a simple root such that $X_\alpha  v\neq0$
and $\gamma\neq\alpha$, then $\gamma-\alpha$ is a positive root. \label{item:15}
\item If $\alpha$ is a simple root such that $\gamma-\alpha$ is a root 
then there exists $z\in\k$ such that $X_\alpha  v=z\,X_{-\gamma+\alpha}  \hwv$.
\end{enumerate}
\end{lemma}
\begin{proof}
The vector $v$ cannot be a linear combination of the highest weight vectors $v_{\lambda_i}$,
otherwise (since the weights $\lambda_i$ are linearly independent) it would belong to $\ft\cdot \hwv\subset\fg\cdot \hwv$.
Moreover, since $X_\alpha\in\fg_{\hwv}$ for all $\alpha\in\sr$, 
$X_\alpha  v$ is a $\Tad$-eigenvector of weight $\gamma-\alpha$ in $\fg\cdot \hwv$. 
\end{proof}

We first deal with the case where $\gamma$ is a root. Notice that
since $\gamma \in \N\sr$, it is then a positive root. As is well
known, we then also have that $\supp(\gamma)$ is a connected subset of
the Dynkin diagram of $G$. 

\begin{lemma} \label{lem:tadposroots1simple}
If $\gamma$ is a root, which is not simple, then there exist at least two 
distinct simple roots $\alpha$ such that $\gamma-\alpha$ is a root.
\end{lemma}

\begin{proof}
Assume that there exists only one simple root $\alpha$ such that $\gamma-\alpha$ is a root.
By Lemma~\ref{lem:simpluspos}, there exists $z \in \k$ such that 
$X_{\alpha}   v=z\,X_{-\gamma+\alpha} \hwv$. 
Moreover, there exists $z'\in\k^\times$ such that $[X_\alpha,X_{-\gamma}]=z'\,X_{-\gamma+\alpha}$.
Therefore, if we put $z''=z/z'$ then $X_\alpha (v+z''\, X_{-\gamma} \hwv)=0$. 
Since $[v] = [v+z''\, X_{-\gamma} \hwv]$ in $\Vg$ we
can assume that $X_{\alpha}   v = 0$. Since $\gamma - \alpha'$ is
not a positive root for all $\alpha' \in \sr\setminus \{\alpha\}$, it then
follows that $X_\alpha  v=0$ for all $\alpha\in\sr$, which contradicts Lemma~\ref{lem:simpluspos}(\ref{item:16}). 
\end{proof}

\begin{proposition}\label{prop:highestshortroot}
If $\gamma$ is a root, of which the support is not of type $\ssG_2$, then it is a locally dominant short root, 
i.e.\ the dominant short root in the root subsystem generated by the simple roots of its support.
\end{proposition}

\begin{proof}
{\bf I}. Let $\alpha_1$ and $\alpha_2$ be two orthogonal simple roots such that 
$\gamma-\alpha_1$ and $\gamma-\alpha_2$ are roots. 
Notice that $\gamma-\alpha_1-\alpha_2$ is also a root. 
We claim that if there exists $\lambda\in F$ not orthogonal to
$\gamma-\alpha_1-\alpha_2$, then
we can assume 
\begin{equation}\label{eqn:simultaneous1}
X_{\alpha_1}   v = X_{\alpha_2}   v =0.
\end{equation}
Indeed, there exist $z_1, z_2 \in \k^\times$ such that
\begin{align*}
[X_{\alpha_1},X_{-\gamma}]=z_1\,X_{-\gamma+\alpha_1}; \\
[X_{\alpha_2},X_{-\gamma}]=z_2\,X_{-\gamma+\alpha_2}.
\end{align*}
Moreover, using the Jacobi identity and the fact that $[X_{\alpha_1},
X_{\alpha_2}]=0$ one finds that
\[
[X_{\alpha_2}, X_{-\gamma+\alpha_1}] = \frac{z_2}{z_1}[X_{\alpha_1}, X_{-\gamma+\alpha_2}]. 
\]
By Lemma~\ref{lem:simpluspos}(3), there exist $z'_1, z'_2 \in \k$ such that
\begin{align*}
X_{\alpha_1}  v=z'_1\,X_{-\gamma+\alpha_1}  \hwv; \\
X_{\alpha_2}  v=z'_2\,X_{-\gamma+\alpha_2}  \hwv.
\end{align*} 
Since 
$X_{\alpha_2}  X_{\alpha_1}  v=X_{\alpha_1}  X_{\alpha_2}  v$ we
obtain that   
\[(\frac{z_2}{z_1}\, z'_1  - z'_2)  [X_{\alpha_1}, X_{-\gamma+\alpha_2}]  \hwv = 0.\]
Using that there exists $\lambda\in F$ not orthogonal to the root $\gamma-\alpha_1-\alpha_2$ 
it follows that $\frac{z_2}{z_1}\, z'_1  - z'_2 = 0$, that is
\[\frac{z'_1}{z_1} = \frac{z'_2}{z_2}.\]
This implies that by replacing $v$ by $v-\frac{z'_1}{z_1}
X_{-\gamma}  \hwv = v-\frac{z'_2}{z_2} X_{-\gamma}  \hwv$, we can
assume (\ref{eqn:simultaneous1}).   

{\bf II}. The same can be done if we have $\alpha_1,\alpha_2,\ldots,\alpha_k$ simple roots 
with $\alpha_j$ orthogonal to $\alpha_{j+1}$ for all $j\in \{1,2,\ldots,k-1\}$ 
and such that $\gamma-\alpha_j$ is a root for all $j\in \{1,2,\ldots,k\}$.
More precisely, we claim that if there exists $\lambda\in F$ not
orthogonal to $\gamma-\alpha_1-\ldots-\alpha_k$, then
we can assume that for all $j\leq k$
\begin{equation}\label{eqn:simultaneous2}
X_{\alpha_j}   v =0.
\end{equation}
Indeed, for every $j\in \{1,2,\ldots, k\}$ there exists, as in part I,
$z_j\in \k^\times$ and $z'_j \in \k$ such that
$[X_{\alpha_j},X_{-\gamma}]=z_j\,X_{-\gamma+\alpha_j}$ and $X_{\alpha_j}  v=z'_j\,X_{-\gamma+\alpha_j}  \hwv$. 
Let $\lambda$ be an element of $F$ that is not orthogonal to
$\gamma-\alpha_1-\ldots-\alpha_k$. 
Then $\lambda$ is not orthogonal to $\gamma-\alpha_j-\alpha_{j+1}$ for
all $j\in \{1,2,\ldots,k-1\}$. By applying part I $(k-1)$ times to the
pairs $\alpha_j$, $\alpha_{j+1}$ we obtain that
\[\frac{z'_1}{z_1} = \frac{z'_2}{z_2} = \ldots = \frac{z'_k}{z_k}.\]
This implies that by replacing $v$ by $v-\frac{z'_1}{z_1}
X_{-\gamma}  \hwv $, we can assume (\ref{eqn:simultaneous2}).   

{\bf III}. 
Assume that there exist more than two simple roots, say $\alpha_1,\ldots,\alpha_k$, 
such that $\gamma-\alpha_j$ is a root for all $j\in\{1,2,\ldots,k\}$. We claim that they can be reordered such that  
$\alpha_j$ is orthogonal to $\alpha_{j+1}$ for all $j<k$ as in part II.

This can be verified by making use of the classification of root
systems, checking case-by-case all the positive roots, 
noticing along the way (although we will not need this) that $k$ is at most $3$. This is straightforward for the classical types. 
To avoid the large number of case-by-case checkings in the exceptional types $\ssE_6$, $\ssE_7$, $\ssE_8$ and $\ssF_4$ 
one can use for example the following argument.
If it were not possible to reorder the simple roots $\alpha_1,\ldots,\alpha_k$ as required, 
then there would exist three roots among them, 
say $\alpha_{j_1},\alpha_{j_2},\alpha_{j_3}$, such that $\alpha_{j_2}$ is not orthogonal to both 
$\alpha_{j_1}$ and $\alpha_{j_3}$. We will now show that this is
impossible for each exceptional type using well-known
properties of root systems of rank $2$ and $3$. Notice, in particular,
that if the support of $\gamma$ is not of type $\ssG_2$ and if
$\gamma-\alpha$ is a root for some simple root $\alpha$, then \[\<\alpha^\vee,\gamma\>\geq0\] since
otherwise there would exist a root string of length greater than $3$.

In types $\ssE_6$, $\ssE_7$ and $\ssE_8$ all the roots have the same length so we would necessarily have 
$\<(\alpha_{j_m})^\vee,\gamma\>=1$ for $m \in \{1,2,3\}$,  
but this is absurd since it would mean that $\<(\alpha_{j_1}+\alpha_{j_2}+\alpha_{j_3})^\vee,\gamma\>=3$.
In type $\ssF_4$ the three simple roots would generate a root subsytem of type $\ssB_3$ or of type $\ssC_3$.
In the former case (type $\ssB_3$) we would necessarily have
$\<(\alpha_{j_1})^\vee,\gamma\>=\<(\alpha_{j_2})^\vee,\gamma\>=1$ assuming $\alpha_{j_1}$ and $\alpha_{j_2}$ are long,
but this is absurd since it would mean $\<(\alpha_{j_1}+\alpha_{j_2}+\alpha_{j_3})^\vee,\gamma\>\geq4$.
In the latter case (type $\ssC_3$) we would necessarily have
$\<(\alpha_{j_1})^\vee,\gamma\>=1$ assuming $\alpha_{j_1}$ is long. 
If $\<(\alpha_{j_3})^\vee,\gamma\>$ is positive, then
$\<(\alpha_{j_1}+\alpha_{j_2}+\alpha_{j_3})^\vee,\gamma\>$ is greater
than $2$, which is not possible in type $\ssF_4$. If
$\<(\alpha_{j_3})^\vee,\gamma\>=0$, then $\gamma + \alpha_{j_3}$ is a
root, and $\<(\alpha_{j_1}+\alpha_{j_2}+\alpha_{j_3})^\vee,\gamma+\alpha_3\>$ is greater than
$2$, which is again absurd.

{\bf IV}.
We now want to prove that $\gamma$ is locally dominant (if the support of $\gamma$ is not of type $\ssG_2$).
The fact that $\gamma$ is locally short then follows.
Indeed, if the support of $\gamma$ is not simply laced, then the
highest root in the root system generated by that support 
does not satisfy Lemma~\ref{lem:tadposroots1simple}:
\begin{itemize}
\item[-] in type $\ssB_n$, $n\geq2$, the highest root is $\alpha_1+2(\alpha_2+\ldots+\alpha_n)=\omega_2$;
\item[-] in type $\ssC_n$, $n\geq3$, the highest root is $2(\alpha_1+\ldots+\alpha_{n-1})+\alpha_n=2\omega_1$;
\item[-] in type $\ssF_4$ the highest root is $2\alpha_1+3\alpha_2+4\alpha_3+2\alpha_4=\omega_1$.
\end{itemize}

To obtain a condradiction we assume that $\gamma$ is not locally
dominant, that is, we assume that there exists $\beta \in
\supp(\gamma)$ such that $\<\beta^{\vee}, \gamma\> <0$. Recall from part III that in type different from $\ssG_2$ if
$\gamma-\alpha$ is a root for a simple root $\alpha$, then $\<\alpha^\vee,\gamma\>\geq0$.

Suppose first that there are exactly $k>2$ simple roots, say $\alpha_1,\ldots,\alpha_k$, 
such that $\gamma-\alpha_j$ is a root for all $j\leq k$. From the
assumption that $\gamma$ is not locally dominant, it follows
that there exists $\lambda\in F$ not orthogonal to
$\gamma-\alpha_1-\ldots-\alpha_k$. By parts II and III we can then assume that $X_{\alpha_j}v=0$ for all $j\leq k$. 
This contradicts Lemma~\ref{lem:simpluspos}(1).

If there are exactly two simple roots $\alpha_1$ and $\alpha_2$ such that
$\gamma-\alpha_1$ and $\gamma-\alpha_2$ are roots, and $\alpha_1$ and $\alpha_2$ are orthogonal, 
then by part I we get the same contradiction with Lemma~\ref{lem:simpluspos}(1).

Furthermore, if the support of $\gamma$ has cardinality $\leq 2$, 
then the proposition follows by Lemma~\ref{lem:tadposroots1simple}.
Indeed, the only roots with support of cardinality $\leq2$ satisfying Lemma~\ref{lem:tadposroots1simple} are:
\begin{itemize}
\item[-] with support of type $\ssA_1$, $\alpha_1$,
\item[-] with support of type $\ssA_2$, $\alpha_1+\alpha_2$,
\item[-] with support of type $\ssB_2$, $\alpha_1+\alpha_2$. 
\end{itemize}

Therefore, we now restrict to the case of support of $\gamma$ of cardinality $>2$,
and assume that there are only two simple roots $\alpha_1$ and $\alpha_2$, such that
$\gamma-\alpha_1$ and $\gamma-\alpha_2$ are roots, and that $\alpha_1$
and $\alpha_2$ are not orthogonal. Notice that $\alpha_1+\alpha_2$ is
a root. Up to exchanging $\alpha_1$ and $\alpha_2$ 
we can assume that 
\begin{equation}\label{eqn:typeclike}
\<\alpha_2^\vee,\gamma\>>0\mbox{ and }\alpha_1+2\alpha_2\notin\rs.
\end{equation}
Indeed, at least one of the two $\< \alpha_1^\vee,\gamma\>$ and  $\< \alpha_2^\vee,\gamma\>$ must be positive
(otherwise $\gamma$ would be antidominant),
and not both $2\alpha_1+\alpha_2$ and $\alpha_1+2\alpha_2$ can be roots. 
If say $2\alpha_1+\alpha_2$ is a root,
then $\|\alpha_1\|<\|\alpha_2\|$, hence $\alpha_2$ is long and therefore $\<\alpha_2^\vee,\gamma\>$ must be $>0$.

Under (\ref{eqn:typeclike}) we have 
\[\<\alpha_2^\vee,\gamma-\alpha_1\>\geq1+1\]
hence  $\gamma-\alpha_1-\alpha_2$ and $\gamma-\alpha_1-2\alpha_2$ are
roots. Since $\gamma$ is not locally dominant, there is an
  element $\lambda$ of $F$ such that $\<(\gamma-\alpha_1-2\alpha_2)^\vee,
  \lambda\>\neq 0$. 

To conclude the proof of the proposition, we use once again an argument similar to that of part
I. Indeed, we will show in part V that we can assume that
$X_{\alpha_1} v = X_{\alpha_2} v = 0$, which contradicts
Lemma~\ref{lem:simpluspos}(\ref{item:16}). 

{\bf V}. We finish by proving the following claim: if $\alpha_1$ and
$\alpha_2$ are simple roots such that
\begin{itemize}
\item[-] $\alpha_1 + 2\alpha_2$ is not a root;
\item[-] $\gamma-\alpha_1$, $\gamma-\alpha_2$,
  $\gamma-\alpha_1-\alpha_2$, and $\gamma-\alpha_1-2\alpha_2$ are
  roots; and
\item[-] $\<(\gamma-\alpha_1-2\alpha_2)^\vee,
  \lambda\>\neq 0$ for some $\lambda \in F$; then
\end{itemize}
we can assume that $X_{\alpha_1} v = X_{\alpha_2} v = 0$.

Since $\alpha_1 + 2\alpha_2$ is not a root we have that
%\begin{equation} 
\([X_{\alpha_2}, X_{\alpha_1+\alpha_2}] = 0. \)  %\label{eq:comm1.1}
%\end{equation}
By the third assumption of the claim, 
\begin{equation}
X_{-(\gamma-\alpha_1-2\alpha_2)} \hwv \neq 0. \label{eq:contr.1}
\end{equation}
We first show that we can assume that 
\begin{equation} \label{eq:1.1}
X_{\alpha_2} v =X_{\alpha_1+\alpha_2}v = 0. 
\end{equation}
There exist $z'_1,z'_2 \in \k$ such that
\begin{align*}
X_{\alpha_2}v &= z'_1 X_{-(\gamma-\alpha_2)} \hwv; \\
X_{\alpha_1+\alpha_2}v &= z'_2 X_{-(\gamma-\alpha_1-\alpha_2)} \hwv.
\end{align*}
Next, there exist $z_1,z_2 \in \k^\times$ such that
\begin{align*}
[X_{\alpha_2},X_{-\gamma}] &= z_1X_{-(\gamma-\alpha_2)};\\
[X_{\alpha_1+\alpha_2},X_{-\gamma}] &= z_2X_{-(\gamma-\alpha_1-\alpha_2)}.
\end{align*}
As in part I, one deduces from $X_{\alpha_2} X_{\alpha_1+\alpha_2} v =
X_{\alpha_1+\alpha_2}  X_{\alpha_2} v$
that
\[(\frac{z_2}{z_1}\,
z'_1-z'_2)[X_{\alpha_2},X_{-(\gamma-\alpha_1-\alpha_2)}]\hwv = 0.\]
Using
 (\ref{eq:contr.1}), it follows that 
\begin{equation}
\frac{z'_1}{z_1} = \frac{z'_2}{z_2}. %\label{eq:5.1}
\end{equation}
Hence, if we replace $v$ by $v-\frac{z'_1}{z_1}
X_{-\gamma} \hwv = v-\frac{z'_2}{z_2}
X_{-\gamma} \hwv $, then equations (\ref{eq:1.1}) hold.  

We now complete the proof by showing that (\ref{eq:1.1}) implies that 
\begin{equation}
X_{\alpha_1} v =0. \label{eq:6.1}
\end{equation}
There exists $z \in \k$ such that $X_{\alpha_1}v = z
X_{-(\gamma-\alpha_1)} \hwv$. From (\ref{eq:1.1}) we have that
\begin{equation*}
0 = X_{\alpha_1+ \alpha_2} v = X_{\alpha_2} X_{\alpha_1} v
= z X_{\alpha_2} X_{-(\gamma-\alpha_1)} \hwv
= z X_{-(\gamma - \alpha_1-\alpha_2)} \hwv,
\end{equation*}
where the second equality uses that $X_{\alpha_2} v =0$ and the fourth
one uses that $X_{\alpha_2} \hwv = 0$. Since equation
(\ref{eq:contr.1}) implies that $X_{\alpha_2} X_{-(\gamma -
  \alpha_1-\alpha_2)} \hwv \neq 0$, we have that
$X_{-(\gamma - \alpha_1-\alpha_2)} \hwv \neq 0$, and therefore that $z=0$
which proves equation (\ref{eq:6.1}), the claim at the start of part V and the proposition.
\end{proof}

The following is Theorem~\ref{thm:vggweights} for the case that
$\gamma$ is a root. 
\begin{corollary} \label{cor:Vggweightposroot}
Let $\gamma$ be a $\Tad$-weight in $\Vgg$. If $\gamma$ is a root, 
then $\gamma$ is a spherically closed spherical root of $G$. 
\end{corollary}
\begin{proof}
If the support of $\gamma$ is not of type $\ssG_2$, then
by Proposition~\ref{prop:highestshortroot} we have only to check the locally dominant short roots.
The following roots do not satisfy Lemma~\ref{lem:tadposroots1simple}.
\begin{itemize}
\item[-] With support of type $\ssD_n$, $n\geq4$: $\alpha_1+2(\alpha_2+\ldots+\alpha_{n-2})+\alpha_{n-1}+\alpha_n=\omega_2$.
\item[-] With support of type $\ssE_6$: $\alpha_1+2\alpha_2+2\alpha_3+3\alpha_4+2\alpha_5+\alpha_6=\omega_2$.
\item[-] With support of type $\ssE_7$: $2\alpha_1+2\alpha_2+3\alpha_3+4\alpha_4+3\alpha_5+2\alpha_6+\alpha_7=\omega_1$.
\item[-] With support of type $\ssE_8$: $2\alpha_1+3\alpha_2+4\alpha_3+6\alpha_4+5\alpha_5+4\alpha_6+3\alpha_7+2\alpha_8=\omega_8$.
\end{itemize}
Therefore, we are left with all spherically closed spherical roots.
\begin{itemize}
\item[-] With support of type $\ssA_n$, $n\geq1$:
  $\alpha_1+\ldots+\alpha_n% = \om_1 + \om_n
  $.
\item[-] With support of type $\ssB_n$, $n\geq2$:
  $\alpha_1+\ldots+\alpha_n% = \om_1
  $.
\item[-] With support of type $\ssC_n$, $n\geq3$:
  $\alpha_1+2(\alpha_2+\ldots+\alpha_{n-1})+\alpha_n% = \om_2
  $.
\item[-] With support of type $\ssF_4$:
  $\alpha_1+2\alpha_2+3\alpha_3+2\alpha_4% = \om_4
  $.
\end{itemize}
If the support of $\gamma$ is of type $\ssG_2$ the only positive root satisfying Lemma~\ref{lem:tadposroots1simple} is
$\alpha_1+\alpha_2$, which is a spherically closed spherical root.
\end{proof}

Let us now consider the case where $\gamma$ is not a root. 
In contrast to the root case, here we notice the following general fact.

\begin{proposition} \label{prop:onesimple}
Let $\alpha$ be a simple root and let $\beta$ be a non-simple positive root 
such that $\alpha+\beta$ is not a root. 
Then there exists no simple root $\alpha'\neq\alpha$
such that $(\alpha+\beta)-\alpha'$ is a root.
\end{proposition}

\begin{proof}
Assume that there exists a simple root $\alpha'\neq\alpha$ 
such that $\alpha+\beta -\alpha'$ is a root. Since $\beta- \alpha'$ is nonzero, it is a root. 
This follows from the fact that $\alpha + \beta$ is not a root, whence$\langle\alpha^\vee , \beta\rangle\geq0$, 
and so $\langle\alpha^\vee , \alpha+\beta - \alpha'\rangle>0$.
Finally, to deduce that $\alpha+\beta$ is a root (i.e.\ a contradiction), 
one can use for example a saturation argument (see \cite[Lemma~13.4.B]{Hum72}) as follows.

Restrict the adjoint representation to the Levi subalgebra associated with $\alpha$ and $\alpha'$. 
Since $\beta - \alpha'$ is a root, both $\beta$ and $\alpha+\beta - \alpha'$ occur as weights 
in the same irreducible summand, say of highest weight $\lambda$. 
From  $\<\alpha^\vee, \beta\> \geq 0$, we get that
$\<\alpha^\vee, \alpha + \beta\> >0$, and since $\alpha + \beta$ is
not a root, $\<(\alpha')^\vee, \alpha + \beta - \alpha'\> \geq 0$, and
so  $\<(\alpha')^\vee, \alpha + \beta\> >0$. Consequently,
$\alpha+\beta$ is dominant with respect to $\alpha$ and
$\alpha'$. Moreover $\lambda - \alpha-\beta$ is a sum of simple roots,
because $\lambda - \beta$ and $\lambda - (\alpha + \beta - \alpha')$
both belong to $\mathrm{span}_{\N}\{\alpha, \alpha'\}$. This implies that $\alpha+\beta$ is a root.
\end{proof}

Let $\gamma$ be a $\Tad$-weight in $\Vgg$ which is not a root. 
Until Proposition~\ref{prop:loc_highest_root}, 
\emph{we assume that $\gamma$ is not the sum of two orthogonal simple roots,}
so that we can speak of the unique simple root $\alpha$ such that $\gamma-\alpha$ is a root.

\begin{lemma}\label{lem:orthogonality}
Let $\alpha$ be the simple root such that $\gamma-\alpha$ is a root.
If $\gamma\neq2\alpha$ then $\alpha$ is orthogonal to $\gamma-\alpha$. 
\end{lemma}

\begin{proof}  
We can choose a basis of $\fg$ 
\[\{X_{\beta}:\beta\mbox{ root}\}\cup\{\alpha^\vee:\alpha\mbox{ simple root}\}\]
such that $[X_\beta,X_{-\beta}]=\beta^\vee$ for all positive roots $\beta$, and
then for all roots $\beta_1,\beta_2$ denote by $c_{\beta_1,\beta_2}$ the scalar such that $[X_{\beta_1},X_{\beta_2}]=c_{\beta_1,\beta_2}X_{\beta_1+\beta_2}$.
For example, a Chevalley basis does the job (see \cite[Theorem~25.2]{Hum72}).

Since $X_\alpha  v\neq0$, we can assume that $X_\alpha  v=X_{-\gamma+\alpha}\hwv$.
Assume also, to obtain a contradiction, that $\langle\alpha^\vee,\gamma-\alpha\rangle>0$. 
Hence $\gamma-2\alpha$ is a positive root. Since $\gamma$ is not a root, we have
that $X_{\gamma-\alpha}X_\alpha  v=X_\alpha X_{\gamma-\alpha}  v$. 
From the following identities
\begin{eqnarray*}
& X_{\gamma-\alpha}X_\alpha  v & = \frac1{c_{\gamma-2\alpha,\alpha}}[X_{\gamma-2\alpha},X_\alpha]X_\alpha  v=
\frac1{c_{\gamma-2\alpha,\alpha}}[X_{\gamma-2\alpha},X_\alpha]X_{-\gamma+\alpha}  \hwv=\\
& & = \frac1{c_{\gamma-2\alpha,\alpha}}\big(X_{\gamma-2\alpha}[X_\alpha,X_{-\gamma+\alpha}]
-X_\alpha[X_{\gamma-2\alpha},X_{-\gamma+\alpha}]\big)  \hwv=\\
& & = \frac{c_{\alpha,-\gamma+\alpha}}{c_{\gamma-2\alpha,\alpha}}[X_{\gamma-2\alpha},X_{-\gamma+2\alpha}]  \hwv
-\frac{c_{\gamma-2\alpha,-\gamma+\alpha}}{c_{\gamma-2\alpha,\alpha}}[X_\alpha, X_{-\alpha}]  \hwv
\end{eqnarray*}

\begin{eqnarray*}
& X_\alpha X_{\gamma-\alpha}  v & = \frac1{c_{\gamma-2\alpha,\alpha}}X_\alpha[X_{\gamma-2\alpha},X_\alpha]  v=
\frac1{c_{\gamma-2\alpha,\alpha}}X_\alpha[X_{\gamma-2\alpha},X_{-\gamma+\alpha}]  \hwv=\\
& & =\frac{c_{\gamma-2\alpha,-\gamma+\alpha}}{c_{\gamma-2\alpha,\alpha}}[X_\alpha,X_{-\alpha}] \hwv
\end{eqnarray*}
it then follows that 
\begin{equation}\label{eqn:coroot1}
\frac{c_{\alpha,-\gamma+\alpha}}{c_{\gamma-2\alpha,\alpha}}(\gamma-2\alpha)^\vee
-2\frac{c_{\gamma-2\alpha,-\gamma+\alpha}}{c_{\gamma-2\alpha,\alpha}}\alpha^\vee
\end{equation}
takes value zero on all $\lambda\in F$. Since $\gamma \in \Z F$, the
expression (\ref{eqn:coroot1}) takes value zero on $\gamma$, too. 

Actually, the linear combination (\ref{eqn:coroot1}) of coroots
does not depend on the choice of the basis of $\fg$. 
Indeed,
\begin{eqnarray*}
&c_{\gamma-2\alpha,\alpha}(\gamma-\alpha)^\vee&=[[X_{\gamma-2\alpha},X_\alpha],X_{-\gamma+\alpha}]=\\
&&=[X_{\gamma-2\alpha},[X_\alpha,X_{-\gamma+\alpha}]]-[X_\alpha,[X_{\gamma-2\alpha},X_{-\gamma+\alpha}]]=\\
&&=c_{\alpha,-\gamma+\alpha}(\gamma-2\alpha)^\vee-c_{\gamma-2\alpha,-\gamma+\alpha}\alpha^\vee
\end{eqnarray*}
and
\[(\gamma-\alpha)^\vee=
\frac{\|\gamma-2\alpha\|^2}{\|\gamma-\alpha\|^2}(\gamma-2\alpha)^\vee+
\frac{\|\alpha\|^2}{\|\gamma-\alpha\|^2}\alpha^\vee.\]
Therefore, since $(\gamma-2\alpha)^\vee$ and $\alpha^\vee$ are linearly independent, (\ref{eqn:coroot1}) becomes
\begin{equation}
\frac{\|\gamma-2\alpha\|^2}{\|\gamma-\alpha\|^2}(\gamma-2\alpha)^\vee+
2\frac{\|\alpha\|^2}{\|\gamma-\alpha\|^2}\alpha^\vee
\end{equation} 
which is proportional to $\gamma^\vee$. 
Since $\|\gamma\|^2$ is not zero, the expression in
(\ref{eqn:coroot1}) cannot take value zero on $\gamma$, and we have
obtained the desired contradiction. 
\end{proof}

\begin{lemma}[{\cite[Lemma~3.6]{bravi&cupit}}] \label{lem:bc36}
Let $\alpha$ be the simple root such that $\gamma-\alpha$ is a positive root.
If $\gamma-\alpha=\beta_1+\beta_2$ with $\beta_1$ and $\beta_2$ positive roots
then $\alpha+\beta_1$ or $\alpha+\beta_2$ is a root.
\end{lemma}

\begin{proof}
Since $X_\alpha  v\neq0$, we can assume that $X_\alpha  v=X_{-\gamma+\alpha}\hwv$. Next, we claim that if $\alpha+\beta_1\not\in\pr$ then $X_{\beta_2}v = 0$. Indeed, if $X_{\beta_2}v$ were nonzero, then it would be a $\Tad$-weight vector of weight $\alpha+\beta_1$. Since $X_{\beta_2}v \in \fg\cdot\hwv$ it would follow by Proposition~\ref{prop:elemVggfacts}(\ref{item:5}) that $\alpha+\beta_1 \in \pr$. This proves the claim. 
Similarly, if $\alpha+\beta_2\not\in\pr$ then $X_{\beta_1}v=0$.
Therefore, if neither $\alpha+\beta_1$ nor $\alpha+\beta_2$ is a root, then $X_{\gamma-\alpha} v=0$.
Since $\gamma\not\in\pr$, this implies
\[0=X_\alpha X_{\gamma-\alpha} v=X_{\gamma-\alpha} X_\alpha v=X_{\gamma-\alpha} X_{-\gamma+\alpha} \hwv\]
which means $X_{-\gamma+\alpha}\hwv=0$, a contradiction.
\end{proof}

\begin{lemma}\label{lem:lengthandsupport}
Let $\alpha$ be the simple root such that $\gamma-\alpha$ is a root.
Let $\delta$ be a simple root and $k$ an integer $2\leq k\leq4$ such that 
$\gamma-j\alpha-\delta$ is a root for $1\leq j\leq k$,
$j\alpha+\delta$ is a root for $1\leq j<k$, 
but $k\alpha+\delta$ is not a root. 
Then $\gamma-k\alpha$ is orthogonal to every $\lambda\in F$;
and in particular 
\begin{equation} \label{eq:8}
\|\gamma-\alpha\|^2=(k-1)\|\alpha\|^2.
\end{equation} 
\end{lemma}

\begin{proof}
We can choose a basis as in the proof of Lemma~\ref{lem:orthogonality}
and, since $X_\alpha  v\neq0$, we can assume that $X_\alpha  v=X_{-\gamma+\alpha}\hwv$.

First, let us assume also, for simplicity, that $k=2$.
Then one has the following identities. 
\begin{eqnarray*}
& X_{\gamma-\alpha}X_\alpha v & = 
\frac1{c_{\gamma-\alpha-\delta,\delta}}[X_{\gamma-\alpha-\delta},X_\delta]X_{-\gamma+\alpha} \hwv=\\
& & = \frac1{c_{\gamma-\alpha-\delta,\delta}}\big(X_{\gamma-\alpha-\delta}[X_\delta,X_{-\gamma+\alpha}]
-X_\delta[X_{\gamma-\alpha-\delta},X_{-\gamma+\alpha}]\big) \hwv=\\
& & = \frac{c_{\delta,-\gamma+\alpha}}{c_{\gamma-\alpha-\delta,\delta}}[X_{\gamma-\alpha-\delta},X_{-\gamma+\alpha+\delta}] \hwv
-\frac{c_{\gamma-\alpha-\delta,-\gamma+\alpha}}{c_{\gamma-\alpha-\delta,\delta}}[X_\delta,X_{-\delta}] \hwv
\end{eqnarray*}

\begin{eqnarray*}
& X_\alpha X_{\gamma-\alpha} v & =
\frac1{c_{\gamma-\alpha-\delta,\delta}}X_\alpha[X_{\gamma-\alpha-\delta},X_\delta] v=
-\frac1{c_{\gamma-\alpha-\delta,\delta}}X_\alpha X_\delta X_{\gamma-\alpha-\delta} v=\\
& & = -\frac1{c_{\gamma-\alpha-\delta,\delta}c_{\gamma-2\alpha-\delta,\alpha}}
X_\alpha X_\delta [X_{\gamma-2\alpha-\delta},X_\alpha] v=\\
& & = -\frac1{c_{\gamma-\alpha-\delta,\delta}c_{\gamma-2\alpha-\delta,\alpha}}
X_\alpha X_\delta [X_{\gamma-2\alpha-\delta},X_{-\gamma+\alpha}] \hwv=\\
& & = -\frac{c_{\gamma-2\alpha-\delta,-\gamma+\alpha}}{c_{\gamma-\alpha-\delta,\delta}c_{\gamma-2\alpha-\delta,\alpha}}
X_\alpha [X_\delta, X_{-\alpha-\delta}] \hwv=\\
& & = -\frac{c_{\gamma-2\alpha-\delta,-\gamma+\alpha}c_{\delta,-\alpha-\delta}}
{c_{\gamma-\alpha-\delta,\delta}c_{\gamma-2\alpha-\delta,\alpha}}
[X_\alpha, X_{-\alpha}] \hwv
\end{eqnarray*}

We thus find a linear combination of co-roots 
\begin{equation}\label{eqn:coroot2}
\frac{c_{\delta,-\gamma+\alpha}}{c_{\gamma-\alpha-\delta,\delta}}(\gamma-\alpha-\delta)^\vee
-\frac{c_{\gamma-\alpha-\delta,-\gamma+\alpha}}{c_{\gamma-\alpha-\delta,\delta}}\delta^\vee
+\frac{c_{\gamma-2\alpha-\delta,-\gamma+\alpha}c_{\delta,-\alpha-\delta}}
{c_{\gamma-\alpha-\delta,\delta}c_{\gamma-2\alpha-\delta,\alpha}}
\alpha^\vee
\end{equation}
which must take value zero on all $\lambda\in F$. We now compute the coefficients in the above linear combination of coroots, 
showing they do not depend on the choice of the basis of $\fg$. 
Indeed,
\begin{eqnarray*}
&c_{\gamma-\alpha-\delta,\delta}(\gamma-\alpha)^\vee&=[[X_{\gamma-\alpha-\delta},X_\delta],X_{-\gamma+\alpha}]=\\
&&=[X_{\gamma-\alpha-\delta},[X_\delta,X_{-\gamma+\alpha}]]-[X_\delta,[X_{\gamma-\alpha-\delta},X_{-\gamma+\alpha}]]=\\
&&=c_{\delta,-\gamma+\alpha}(\gamma-\alpha-\delta)^\vee-c_{\gamma-\alpha-\delta,-\gamma+\alpha}\delta^\vee
\end{eqnarray*}
and, since
\begin{eqnarray*}
\lefteqn{c_{\gamma-\alpha-\delta,\delta}c_{\gamma-2\alpha-\delta,\alpha}(\gamma-\alpha)^\vee
+c_{\gamma-\alpha-\delta,-\gamma+\alpha}c_{\gamma-2\alpha-\delta,\alpha}\delta^\vee=}\\
&&=[[[X_{\gamma-2\alpha-\delta},X_{\alpha}],X_\delta],X_{-\gamma+\alpha}]
-[[[X_{\gamma-2\alpha-\delta},X_\alpha],X_{-\gamma+\alpha}],X_{\delta}]=\\
&&=[[X_{\gamma-2\alpha-\delta},X_\alpha],[X_\delta,X_{-\gamma+\alpha}]]=\\
&&=[[X_{\gamma-2\alpha-\delta},[X_\delta,X_{-\gamma+\alpha}]],X_\alpha]
-[X_{\gamma-2\alpha-\delta},[[X_\delta,X_{-\gamma+\alpha}],X_\alpha]]=\\
&&=[[X_\delta,[X_{\gamma-2\alpha-\delta},X_{-\gamma+\alpha}]],X_\alpha]
-[X_{\gamma-2\alpha-\delta},[[X_\delta,X_{-\gamma+\alpha}],X_\alpha]]=\\
&&=-c_{\delta,-\alpha-\delta}c_{\gamma-2\alpha-\delta,-\gamma+\alpha}\alpha^\vee
-c_{-\gamma+\alpha+\delta,\alpha}c_{\delta,-\gamma+\alpha}(\gamma-2\alpha-\delta)^\vee,
\end{eqnarray*}
also
\[
c_{-\gamma+\alpha+\delta,\alpha}c_{\delta,-\gamma+\alpha}(\gamma-2\alpha-\delta)^\vee=
-c_{\gamma-2\alpha-\delta,\alpha}c_{\delta,-\gamma+\alpha}(\gamma-\alpha-\delta)^\vee
-c_{\delta,-\alpha-\delta}c_{\gamma-2\alpha-\delta,-\gamma+\alpha}\alpha^\vee.
\]

On the other hand,
\[(\gamma-\alpha)^\vee=
\frac{\|\gamma-\alpha-\delta\|^2}{\|\gamma-\alpha\|^2}(\gamma-\alpha-\delta)^\vee+
\frac{\|\delta\|^2}{\|\gamma-\alpha\|^2}\delta^\vee\]
and
\[(\gamma-2\alpha-\delta)^\vee=
\frac{\|\gamma-\alpha-\delta\|^2}{\|\gamma-2\alpha-\delta\|^2}(\gamma-\alpha-\delta)^\vee
-\frac{\|\alpha\|^2}{\|\gamma-2\alpha-\delta\|^2}\alpha^\vee.\]

Therefore, since $(\gamma-\alpha-\delta)^\vee$ is neither proportional to $\delta^\vee$ 
nor to $\alpha^\vee$, 
(\ref{eqn:coroot2}) becomes
\begin{equation}
\frac{\|\gamma-\alpha-\delta\|^2}{\|\gamma-\alpha\|^2}(\gamma-\alpha-\delta)^\vee
+\frac{\|\delta\|^2}{\|\gamma-\alpha\|^2}\delta^\vee
-\frac{\|\alpha\|^2}{\|\gamma-\alpha\|^2}\alpha^\vee
\end{equation} 
which is proportional to $(\gamma-2\alpha)^\vee$. 

For $k>2$ the proof is similar. If $k=3$, the analog of (\ref{eqn:coroot2}) is
\begin{eqnarray*}
&&\frac{c_{\delta,-\gamma+\alpha}}{c_{\gamma-\alpha-\delta,\delta}}(\gamma-\alpha-\delta)^\vee
-\frac{c_{\gamma-\alpha-\delta,-\gamma+\alpha}}{c_{\gamma-\alpha-\delta,\delta}}\delta^\vee+\\
&+&\frac{c_{\gamma-2\alpha-\delta,-\gamma+\alpha}c_{\delta,-\alpha-\delta}}
{c_{\gamma-\alpha-\delta,\delta}c_{\gamma-2\alpha-\delta,\alpha}}\alpha^\vee+\\
&-&\frac{c_{\gamma-3\alpha-\delta,-\gamma+\alpha}c_{\delta,-\alpha-\delta}c_{\alpha,-2\alpha-\delta}}
{c_{\gamma-\alpha-\delta,\delta}c_{\gamma-2\alpha-\delta,\alpha}c_{\gamma-3\alpha-\delta,\alpha}}
\alpha^\vee
\end{eqnarray*}
which is proportional to $(\gamma-3\alpha)^\vee$. If $k=4$, we get
\begin{eqnarray*}
&&\frac{c_{\delta,-\gamma+\alpha}}{c_{\gamma-\alpha-\delta,\delta}}(\gamma-\alpha-\delta)^\vee
-\frac{c_{\gamma-\alpha-\delta,-\gamma+\alpha}}{c_{\gamma-\alpha-\delta,\delta}}\delta^\vee+\\
&+&\frac{c_{\gamma-2\alpha-\delta,-\gamma+\alpha}c_{\delta,-\alpha-\delta}}
{c_{\gamma-\alpha-\delta,\delta}c_{\gamma-2\alpha-\delta,\alpha}}\alpha^\vee+\\
&-&\frac{c_{\gamma-3\alpha-\delta,-\gamma+\alpha}c_{\delta,-\alpha-\delta}c_{\alpha,-2\alpha-\delta}}
{c_{\gamma-\alpha-\delta,\delta}c_{\gamma-2\alpha-\delta,\alpha}c_{\gamma-3\alpha-\delta,\alpha}}\alpha^\vee+\\
&+&\frac{c_{\gamma-4\alpha-\delta,-\gamma+\alpha}c_{\delta,-\alpha-\delta}c_{\alpha,-2\alpha-\delta}c_{\alpha,-3\alpha-\delta}}
{c_{\gamma-\alpha-\delta,\delta}c_{\gamma-2\alpha-\delta,\alpha}c_{\gamma-3\alpha-\delta,\alpha}c_{\gamma-4\alpha-\delta,\alpha}}\alpha^\vee
\end{eqnarray*}
which is proportional to $(\gamma-4\alpha)^\vee$.

Finally, since $\gamma-k\alpha$ is orthogonal to every $\lambda\in F$, 
we have $(\gamma-k\alpha,\gamma)=0$, which yields (\ref{eq:8}).
Indeed,
the assumption implies that $\gamma \neq 2\alpha$, 
hence $(\alpha,\gamma-\alpha)=0$ by Lemma~\ref{lem:orthogonality},
and 
\[
0=(\gamma-k\alpha,\gamma)=\|\gamma-\alpha\|^2-(k-1)\|\alpha\|^2.
\]
\end{proof}

\begin{proposition}\label{prop:loc_highest_root}
Suppose $\gamma$ is not a root and let $\alpha$ be a simple root such that $\gamma-\alpha$ is a root. 
Then $\gamma-\alpha$ is locally the highest root, 
i.e.\ the highest root in the root subsystem generated by the simple roots of its support.
\end{proposition}

\begin{proof}
{\bf I}.
First we want to prove that $\gamma-\alpha$ is locally dominant. 
We can assume that $\gamma-\alpha$ is not simple.
Hence, by Lemma~\ref{lem:orthogonality}, $\alpha$ is orthogonal to $\gamma-\alpha$.

There exists a simple root $\delta$ (different from $\alpha$) such that $\gamma-\alpha-\delta$ is a root.
By Proposition~\ref{prop:onesimple} and Lemma~\ref{lem:bc36} $\alpha+\delta$ is a root.

Since $\alpha+\delta$ is a root, $\<\alpha^\vee, \delta\> < 0$. Therefore,
$\<\alpha^\vee,\gamma-\alpha-\delta\>>0$ hence $\gamma-2\alpha-\delta$ is a root.
If moreover $2\alpha+\delta$ is a root, then by $\sl(2)$-theory,
$\<\alpha^\vee, \alpha + \delta\> \leq 0$ and so 
$\<\alpha^\vee,\gamma-2\alpha-\delta\>\geq0$,  whence $\gamma-3\alpha-\delta$ is a root.
If $3\alpha+\delta$  is also a root, then $\alpha$ and $\delta$ span a
root system of type $\ssG_2$. Consequently,
$\<\alpha^\vee,\gamma-3\alpha-\delta\>=-1$ and $\gamma-4\alpha-\delta$ is a root.

Therefore we can apply Lemma~\ref{lem:lengthandsupport} and obtain that, for some $k\geq1$, 
$\gamma-k\alpha$ is orthogonal to every $\lambda\in F$. This implies
that $\< (\alpha')^\vee, \gamma \> = 0$ for all $\alpha' \in
\supp(\gamma)\setminus\{\alpha\}$, whence $\< (\alpha')^\vee,
\gamma-\alpha \> \geq 0$ for all such $\alpha'$. Since $\alpha$ is
orthogonal to $\gamma-\alpha$, it follows that $\gamma-\alpha$
is locally dominant. 

{\bf II}.
To obtain a contradiction, we now assume that $\gamma-\alpha$ is not locally the highest root, 
that is, a locally short dominant root 
with support of non-simply-laced type:
\begin{itemize}
\item[-] in type $\ssB_n$, $n\geq2$, the short dominant root is $\alpha_1+\ldots+\alpha_n=\omega_1$;
\item[-] in type $\ssC_n$, $n\geq3$, the short dominant root is $\alpha_1+2(\alpha_2+\ldots+\alpha_{n-1})+\alpha_n=\omega_2$;
\item[-] in type $\ssF_4$ the short dominant root is $\alpha_1+2\alpha_2+3\alpha_3+2\alpha_4=\omega_4$;
\item[-] in type $\ssG_2$ the short dominant root is $2\alpha_1+\alpha_2=\omega_1$.
\end{itemize}
By equation (\ref{eq:8}), $\alpha$ is also short and $k=2$, in particular the support of $\gamma$ is not of type $\ssG_2$.
Moreover, by Lemma~\ref{lem:orthogonality}, $\alpha$ is orthogonal to $\gamma-\alpha$.
In type $\ssB_n$ and in type $\ssF_4$ this implies that $\gamma$ is a root.

We are left with the case where the support of $\gamma-\alpha$ is of type $\ssC_n$. 
Since $\alpha$ is short, $\alpha$ is orthogonal to $\gamma-\alpha$, 
$\gamma$ is not a root,
and moreover there exists a simple root $\delta\neq\alpha$ satisfying the hypothesis of Lemma~\ref{lem:lengthandsupport} for $k=2$,
we have that $n>3$, $\delta=\alpha_2$ and $\alpha=\alpha_3$. 
This contradicts Lemma~\ref{lem:bc36}, 
because $\alpha_1$ and $\gamma-\alpha-\alpha_1$ are roots, 
but neither $\alpha_1+\alpha$ nor $\gamma-\alpha_1$ is a root. 
\end{proof}

The following is Theorem~\ref{thm:vggweights} for the case that
$\gamma$ is not a root. 
\begin{corollary} \label{cor:Vggweightnotposroot}
Let $\gamma$ be a $\Tad$-weight in $\Vgg$. If $\gamma$ is not a 
root, then $\gamma$ is a spherically closed spherical root of $G$. 
\end{corollary}
\begin{proof}
We list all the locally highest roots $\beta$ 
and deduce which are the only possible non-roots $\gamma$ (obtained by adding to $\beta$ a simple root)
satisfying Lemmas~\ref{lem:orthogonality},~\ref{lem:bc36} and~\ref{lem:lengthandsupport}.

In general, $\<\alpha^\vee,\beta\>$ must be $\geq0$ otherwise $\alpha+\beta\in\pr$.
If $\alpha$ is not in the support of $\beta$ it must be orthogonal to $\beta$, 
and in this case, by Lemma~\ref{lem:bc36}, $\beta$ must necessarily be simple. 

Let us start with $\beta$ simple, i.e., with support of type $\ssA_1$: $\beta=\alpha_1=2\omega_1$ 
gives only \[2\alpha_1\] or \[\alpha_1+\alpha'_1.\]

Let us now pass to $\beta$ not simple and recall that $\alpha$ must necessarily belong to the support of $\beta$,
moreover by Lemma~\ref{lem:orthogonality} $\<\alpha^\vee,\beta\>=0$ 
and by Lemma~\ref{lem:lengthandsupport}, for all $\alpha'\in S\setminus\{\alpha\}$, 
$\<(\alpha')^\vee,\alpha+\beta\>=0$.

With support of type $\ssA_n$, $n\geq2$: $\beta=\alpha_1+\ldots+\alpha_n=\omega_1+\omega_n$ 
gives only, for $n=3$, \[\alpha_1+2\alpha_2+\alpha_3.\]

With support of type $\ssB_n$, $n\geq2$: $\beta=\alpha_1+2(\alpha_2+\ldots+\alpha_n)=\omega_2$ if $n\geq3$ (it equals $2\omega_2$ if $n=2$)
gives only \[2(\alpha_1+\ldots+\alpha_n)\] or, for $n=3$, \[\alpha_1+2\alpha_2+3\alpha_3.\]

With support of type $\ssD_n$, $n\geq4$: $\beta=\alpha_1+2(\alpha_2+\ldots+\alpha_{n-2})+\alpha_{n-1}+\alpha_n=\omega_2$
gives only \[2(\alpha_1+\ldots+\alpha_{n-2})+\alpha_{n-1}+\alpha_n\] 
or, for $n=4$, \[\alpha_1+2\alpha_2+2\alpha_3+\alpha_4\] and \[\alpha_1+2\alpha_2+\alpha_3+2\alpha_4\] 
which are equal to $2\alpha_1+2\alpha_2+\alpha_3+\alpha_4$
up to an automorphism of the Dynkin diagram.

With support of type $\ssG_2$: $\beta=3\alpha_1+2\alpha_2=\omega_2$ 
gives only \[4\alpha_1+2\alpha_2.\]

The remaining cases give no other possibilities:
\begin{itemize}
\item[-]with support of type $\ssC_n$, $n\geq3$, $\beta=2(\alpha_1+\ldots+\alpha_{n-1})+\alpha_n=2\omega_1$;
\item[-]with support of type $\ssE_6$, $\beta=\alpha_1+2\alpha_2+2\alpha_3+3\alpha_4+2\alpha_5+\alpha_6=\omega_2$;
\item[-]with support of type $\ssE_7$, $\beta=2\alpha_1+2\alpha_2+3\alpha_3+4\alpha_4+3\alpha_5+2\alpha_6+\alpha_7=\omega_1$;
\item[-]with support of type $\ssE_8$, $\beta=2\alpha_1+3\alpha_2+4\alpha_3+6\alpha_4+5\alpha_5+4\alpha_6+3\alpha_7+2\alpha_8=\omega_8$;
\item[-]with support of type $\ssF_4$, $\beta=2\alpha_1+3\alpha_2+4\alpha_3+2\alpha_4=\omega_1$.
\end{itemize} 
\end{proof}

\subsection{Further properties of $\Tad$-weights in $\Vgg$} \label{sec:vggweightsprops}
After Theorem~\ref{thm:vggweights} the only possible $\Tad$-weights in $\Vgg$ are 
spherically closed spherical roots of $G$,
but each of them occur only under special conditions which we are going to describe.

The first statement is indeed a refinement of Theorem~\ref{thm:vggweights}. 
Recall the notion of compatibility with $S^p$ (see axiom~(S) of Definition~\ref{def:spherical-systems} 
and Remark~\ref{rem:spherical-systems}.1).

\begin{theorem}\label{thm:vggweightsrefinement}
If $\gamma$ is a $\Tad$-weight in $\Vgg$ 
then $\gamma$ is a spherically closed spherical root of $G$ 
compatible with $S^p(\wm)$.
\end{theorem}
\begin{proof}
If $\gamma = \alpha_1 + \alpha_2 + \ldots + \alpha_n$ with support of type $\ssA_n$,
then $\{\alpha_2, \alpha_3,\ldots, \alpha_{n-1}\} \subset S^p(\wm)$.  
This follows from part I of the proof of Proposition~\ref{prop:highestshortroot}.

If $\gamma = \alpha_1 + 2\alpha_2 + \alpha_3$ with support of type $\ssA_3$,
then $\{\alpha_1, \alpha_3\} \subset S^p(\wm)$.  
This follows by Lemma~\ref{lem:lengthandsupport} ($\alpha=\alpha_2$, $\delta=\alpha_1$ and $k=2$).

If $\gamma = \alpha_1 + \alpha_2 + \ldots + \alpha_n$ with support of type $\ssB_n$,
then $\{\alpha_2, \alpha_3,\ldots, \alpha_{n-1}\} \subset S^p(\wm)$ and $\alpha_n\not\in S^p(\wm)$.  
The former follows from part I of the proof of Proposition~\ref{prop:highestshortroot}.
For the latter, we can assume that $X_{\alpha_n}v=0$
and $X_{\alpha_1}v=X_{-\gamma+\alpha_n}\hwv$ nonzero, 
which implies $\alpha_n\not\in S^p$.

If $\gamma = 2(\alpha_1 + \ldots + \alpha_n)$ with support of type $\ssB_n$,
then $\{\alpha_2, \ldots, \alpha_n\} \subset S^p(\wm)$.  
This follows by Lemma~\ref{lem:lengthandsupport} ($\alpha=\alpha_1$, $\delta=\alpha_2$ and $k=2$).

If $\gamma = \alpha_1 + 2\alpha_2 + 3\alpha_3$ with support of type $\ssB_3$,
then $\{\alpha_1, \alpha_2\} \subset S^p(\wm)$.  
This follows by Lemma~\ref{lem:lengthandsupport} ($\alpha=\alpha_3$, $\delta=\alpha_2$ and $k=3$).

If $\gamma = \alpha_1 + 2(\alpha_2+\ldots+\alpha_{n-1})+ \alpha_n$ with support of type $\ssC_n$,
then $\{\alpha_3, \alpha_4, \ldots, \alpha_n\} \subset S^p(\wm)$.
This follows from part V of the proof of Proposition~\ref{prop:highestshortroot}.

If $\gamma = 2(\alpha_1 + \ldots + \alpha_{n-2})+\alpha_{n-1}+\alpha_n$ with support of type $\ssD_n$,
then $\{\alpha_2, \ldots, \alpha_n\} \subset S^p(\wm)$.
This follows by Lemma~\ref{lem:lengthandsupport} ($\alpha=\alpha_1$, $\delta=\alpha_2$ and $k=2$).

If $\gamma = \alpha_1 + 2\alpha_2 + 3\alpha_3 + 2\alpha_4$ with support of type $\ssF_4$,
then $\{\alpha_1, \alpha_2, \alpha_3\} \subset S^p(\wm)$.
This follows from part V of the proof of Proposition~\ref{prop:highestshortroot}.

%If $\gamma = 2\alpha_1 + \alpha_2$ with support of type $\ssG_2$,
%then it never occurs as a $\Tad$-weight in $\Vgg$.
%Indeed, the weight $\gamma$ is a root, but as already observed in the proof of Corollary~\ref{cor:Vggweightposroot} 
%does not fulfill the condition of Lemma~\ref{lem:tadposroots1simple}.

If $\gamma = 4\alpha_1 + 2\alpha_2$ with support of type $\ssG_2$,
then $\alpha_2\in S^p(\wm)$.
This follows by Lemma~\ref{lem:lengthandsupport} ($\alpha=\alpha_1$, $\delta=\alpha_2$ and $k=4$).
\end{proof}

\begin{proposition}\label{prop:multfreequot}
If $\gamma$ is not a simple root then the $\Tad$-eigenspace $\Vgg_{(\gamma)}$ has dimension $\leq1$.
\end{proposition}

\begin{proof}
If $\gamma$ is a root (not simple), recall that there exist two simple roots, say $\alpha_1$ and $\alpha_2$, 
such that $\gamma-\alpha_1$ and $\gamma-\alpha_2$ is a root, 
and $\gamma-\alpha$ is not a root for all $\alpha\in S\setminus\{\alpha_1,\alpha_2\}$.
In particular, for all $\alpha\in S\setminus\{\alpha_1,\alpha_2\}$, we necessarily have $X_\alpha v=0$.
By adding to $v$ a suitable scalar multiple of $X_{-\gamma} \hwv$, 
we can assume that also $X_{\alpha_2} v=0$.
Moreover, by choosing a suitable scalar multiple,
we can assume that $X_{\alpha_1} v=X_{-\gamma+\alpha_1} \hwv$.

If $\gamma$ is neither a root nor the sum of two orthogonal simple roots,
recall that there exists a simple root $\alpha_1$
such that $\gamma-\alpha_1$ is a root,
and $\gamma-\alpha$ is not a root for all $\alpha\in S\setminus\{\alpha_1\}$. 
In particular, for all $\alpha\in S\setminus\{\alpha_1\}$, we necessarily have $X_\alpha v=0$.
Therefore, by choosing a suitable scalar multiple,
we can assume that $X_{\alpha_1} v=X_{-\gamma+\alpha_1} \hwv$.

In both cases we claim that under the above assumptions $v$ is uniquely determined.
Indeed, if $v_1$ and $v_2$ are two vectors in $V$ of $\Tad$-weight $\gamma$ fulfilling the above conditions,
then $X_{\alpha}(v_1-v_2)=0$ for all $\alpha\in S$, which implies $v_1=v_2$.

We are left with only one case: the spherical root $\gamma = \alpha+\alpha'$ with support of type $\ssA_1\times\ssA_1$.
We can assume $X_{\alpha}v=X_{-\alpha'}\hwv$.
For all $i\in\{1,\ldots,r\}$, $\dim V(\lambda_i)_{(\gamma)}\leq1$, 
and the condition $X_{\alpha} v=X_{-\alpha'} \hwv$   
uniquely determines every component $v_i \in V(\lambda_i)$ of $v$.
\end{proof}

\section{The weight spaces of $\tg \hs$} \label{sec:weight-spaces-tH}

In this section we prove the following.
\begin{theorem} \label{thm:TX0hs}
If $\wm$ is a free monoid of dominant weights, then $\tg\hs$ is a
multiplicity-free $\Tad$-module of which all the weights belong to
$\Sigma^{sc}(G)$. Moreover, if $\gamma \in
\Sigma^{sc}(G)$ occurs as a $\Tad$-weight in $\tg \hs$ then
$\gamma$ is N-adapted to $\wm$. 
\end{theorem}
\begin{proof}
The assertion that all $\Tad$-weights of $\tg\hs$ belong to $\Sigma^{sc}(G)$
follows from the inclusion $\tg\hs \into \Vgg$ and
Theorem~\ref{thm:vggweights}, while the assertion that the weight space
$(\tg\hs)_{(\gamma)}$ has dimension at most one follows from
Proposition~\ref{prop:multfreequot} if $\gamma \notin \sr$, and from
Proposition~\ref{prop:tangsr} below if $\gamma \in \sr$. The statement
that if $\gamma \in \Sigma^{sc}(G)$ is a $\Tad$-weight in $\tg\hs$, then
$\gamma$ is N-adapted to $\wm$, is contained in
Proposition~\ref{prop:tangsr} for $\gamma \in \sr$ and is shown in
Section~\ref{sec:nonsimspher} for $\gamma \notin \sr$. 
\end{proof}

Recall from Proposition~\ref{prop:msintohs} that $\ms$ is
$\Tad$-equivariantly isomorphic to an open subscheme of $\hs$. 
Because every $\Tad$-weight in $\tg\ms \isom \tg\hs$ is an element of $\Sigma^{sc}(G)$ (see Theorem~\ref{thm:vggweights}) we obtain the following converse to the second statement in Theorem~\ref{thm:TX0hs}.
\begin{corollary} \label{cor:TX0hs}
Let $\wm$ be a free monoid of dominant weights and let $\sigma \in \Sigma^{sc}(G)$. If $\sigma$ is N-adapted to $\wm$, then $\sigma$ is a $\Tad$-weight in $\tg\ms$. 
\end{corollary}
\begin{proof}
Let $X$ be an affine spherical $G$-variety with $\wm(X)=\wm$ and
$\Sigma^N(X)=\{\sigma\}$, and let $\mathscr M_X$ be its root
monoid. Recall that $\Sigma^N(X)$ is the basis of the saturation of
$\mathscr M_X$. Let $\{a_1, a_2, \ldots, a_k\}$ be a subset of $\N$
such that $\{a_1\sigma, a_2\sigma, \ldots, a_k\sigma\}$ is the minimal
set of generators of $\mathscr M_X$. By~\cite[Proposition 2.13]{alexeev&brion-modaff}, 
the $\Tad$-orbit closure of $X$, seen as a closed point of $\ms$, is $\Spec(\k[-\mathscr M_X])$. 
A straightforward computation using the basic theory of semigroup rings 
(see, e.g., \cite[\S 7.1]{miller-sturmfels}) shows that 
\[\tg(\overline{\Tad\cdot X}) \isom \k_{a_1 \sigma} \oplus \k_{a_2\sigma} \oplus \ldots \oplus \k_{a_k\sigma}\]
as $\Tad$-modules, where we used $\k_{a_i\sigma}$ for the one-dimensional $\Tad$-representation of weight $a_i\sigma$. We claim that one of the $a_i$ is equal to $1$ (and
consequently that $\mathscr M_X$ is generated by $\{\sigma\}$). We
show this by contradiction. Suppose that all of the $a_i$ are at least $2$. Then $k \ge 2$, since otherwise $\sigma$ would not be in $\Z\mathscr M_X$. Since $\tg(\overline{\Tad\cdot X}) \subset \tg\ms \subset \Vgg$, it then follows from Theorem~\ref{thm:vggweights} that $\{\sigma, a_1\sigma, a_2\sigma\} \subset \Sigma^{sc}(G)$. By the classification of spherically closed spherical roots (cf. Proposition~\ref{prop:sigmaG}) this is impossible: only the double or half of a spherically closed spherical root can be a spherically closed spherical root, and never both. 
\end{proof}

As before, $\wm$ will be a free monoid of dominant weights with basis
$F = \{\lambda_1, \lambda_2, \ldots, \lambda_r\}$. If $\lambda \in F$,
then we will write $\lambda^\#$ for the corresponding element of the
dual basis of $(\Z\wm)^*$; in other words, for all $\mu \in F\setminus
\{\lambda\}$ we have $\<\lambda^\#,\mu\>=0$, whereas 
$\<\lambda^\#,\lambda\> =1$.  Recall that $E(\wm)$ is defined in (\ref{eq:1}).
Because $\wm$ is free, we have that $E(\wm)$ is the dual basis to
$F$: \[E(\wm) = \{\lambda^\# \in (\Z\wm)^*\colon \lambda \in F\}.\]
For $\lambda \in F$ we put
\[z_{\lambda} := x_0 - v_{\lambda}.\] 
%If $\lambda = \lambda_i$ for
%some $i \in \{1,2,\ldots,r\}$ then we
%will use $z_i$ for the same vector, that is
%\[z_i : = v_{\lambda_1} + \ldots + \widehat{v_{\lambda_i}}+ \ldots +
%v_{\lambda_r}.\] 

\subsection{The extension criterion}
We recall from~\cite{dsmot-preprint} a criterion which allows to decide whether
a $\Tad$-eigenvector $[v] \in \Vgg \isom H^0(G\cdot X_0,
\shN_{X_0|V})^G$ belongs to the subspace $\tg\hs \isom H^0(X_0, \shN_{X_0|V})^G$.

We denote by  $X_0^{\leq 1}  \subset X_0$ 
the union of $G \cdot x_0$ with all $G$-orbits of $X_0$ that have
codimension $1$.  
By \cite [Lemma~1.14] {brion-cirmactions} 
$X_0^{\leq 1}$ is an open subset of $X_0$.
The following proposition is a special case of \cite[Lemma
3.9]{brion-ihs}. Together with Theorem~\ref{thm:extension-criterion}
it gives the aforementioned criterion. 
\begin{proposition} \label{prop:ext_lemma}
A section $s \in H^0(G\cdot X_0,\shN_{X_0|V})$ extends to $X_0$ if and
only if it extends to $X_0^{\leq 1}$.
\end{proposition}

We recall that the orbit structure of $X_0$ is
well-understood~\cite[Theorem~8]{vin&pop-quasi}. It is easy to
describe the orbits of codimension $1$ (see, e.g.,
\cite[Proposition~3.1]{degensphermod} for details).

\begin{proposition} \label{prop:codim1orbits}
The $G$-orbits of codimension $1$ in $X_0$ are exactly the orbits $G
\cdot z_{\lambda}$ where $\lambda$ is an element of $F$ that satisfies
the following property:
\begin{equation*}
\text{for every } \alpha \in S \text{ such that } 
\<\alpha^{\vee},\lambda\> \neq 0 \text{ there exists } \mu \in F \setminus
\{\lambda\} \text{ such that } \<\alpha^{\vee},\mu\> \neq 0. 
\end{equation*}
\end{proposition}

\begin {theorem}[{\cite[Theorem 2.5]{dsmot-preprint}}]  \label{thm:extension-criterion}
Let $v \in V$ be a $\Tad$-eigenvector of weight $\gamma$ such that
 \(
                0 \not=   [v] \in  \Vgg.
\)  
Let $\lambda \in F$. Recall that $z_{\lambda} = x_0 - v_{\lambda}$.
Assume that  $z_{\lambda} \in X_0^{\leq 1}$ and put $Z:=G\cdot x_0 \cup G\cdot
z_{\lambda}.$ Put $a:=\<\lambda^\#, \gamma\>$. 
Denote by   $s \in  H^0 ( G \cdot x_0, \shN_{X_0|V} )^G$  the
$G$-equivariant section such that $s(x_0) = [v]$. 
\begin{enumerate}[A)]
\item If  $a \leq 0$,  then  $s$ extends to an element of $H^0 (Z,  \shN_{X_0|V} )^G$.
\item If  $a > 1$,   then $s$  does not extend  to an element of $H^0 (Z,  \shN_{X_0|V})^G$.
\item If $a = 1$, then the following are equivalent:
\begin{enumerate}[i)]
 \item $s$  extends to an element of $H^0 (Z,  \shN_{X_0|V} )^G$;
\item there exist  $\hat{v}  \in V(\lambda)$  such that  
                         $[v] = [\hat{v}]$  as elements of  $\Vg$.
\end{enumerate}

\end{enumerate}
\end{theorem}

\subsection{The spherical root $\gamma = \alpha \in \sr$} 
                                
In this section, we discuss the $\Tad$-weight space
$(\tg\hs)_{(\alpha)}$, where $\alpha$ is a simple
root. Specifically, we will prove the following proposition, which is a special case of 
Theorem~\ref{thm:TX0hs}.

\begin{proposition} \label{prop:tangsr}
If $\alpha$ is a simple root then  $\dim (\tg\hs)_{(\alpha)} \leq
1$.  Moreover, if 
$\dim (\tg\hs)_{(\alpha)} = 1$ then $\alpha$ is
$N$-adapted to $\wm$. 
\end{proposition}
The proof of Proposition~\ref{prop:tangsr} will be given on page~\pageref{proofproptangsr}. We first need a few lemmas and introduce notation we
will use for the remainder of this section. Put $F(\alpha) :=
\{\lambda \in F \colon \<\alpha^\vee,\lambda\> \neq 0\}$. We order
the elements of $F$ such that for $F(\alpha) = \{\lambda_1, \lambda_2,
\ldots, \lambda_p\}$ for some $p \leq r$. Then $F \setminus F(\alpha)
= \{\lambda_{p+1}, \lambda_{p+2}, \ldots,\lambda_r\}$. 

\begin{lemma} \label{lem:vggalpha}
For every $i \in \{1,2,\ldots, p\}$, put $v_i = X_{-\alpha} v_{\lambda_i}$. Then
$v_1 + v_2+ \ldots + v_p$ spans the $\Tad$-weight space of weight $\alpha$  in $\fg \cdot x_0$. If $\alpha \in \Z\wm$,
then  
\[\Vgg_{(\alpha)} = \<[v_1], [v_2],\ldots,[v_{p-1}]\>_{\k}.\]
\end{lemma}
\begin{proof}
By elementary highest weight theory, the $\Tad$-weight space in $V$ of weight
$\alpha$ is spanned by $\{v_1,v_2,\ldots,v_p\}$, and the intersection
of this weight space with $\fg \cdot x_0$ is the line spanned by
$X_{-\alpha}x_0 =  v_1 + v_2+ \ldots + v_p$. A straightforward
application of Proposition~\ref{prop:elemVggfacts} shows that $[v_i]
\in \Vgg$ for every $i \in \{1,2,\ldots,p-1\}$. 
\end{proof}

\begin{lemma} \label{lem:codim1orb}
Suppose $\alpha \in \Z\wm$ and $|F(\alpha)| \geq 2$. 
Let $\lambda \in F$. If $\<\lambda^\#,\alpha\> >0$, then $G\cdot
z_{\lambda}$ has codimension $1$ in $X_0$.
\end{lemma}
\begin{proof}
We will apply Proposition~\ref{prop:codim1orbits}.  
Since $\alpha \in \Z\wm$ and $\wm$ is free, there exists a partition $F = F_1 \cup
F_2$ of $F$ and for every $\mu \in F$ a unique nonnegative integer
$a_{\mu}$ such that 
\begin{equation} \label{eq:2}
\alpha = \sum_{\mu \in F_1} a_{\mu} \mu - \sum_{\mu \in F_2} a_{\mu}
\mu.
\end{equation}
By assumption $\lambda \in F_1$ and $a_{\lambda} = \<\lambda^\#,\alpha\> > 0$. 
Let $\beta \in \sr \setminus \{\alpha\}$ such that
$\<\beta^\vee,\lambda\> \neq 0$. Then, since $F \subset \dw$ and
 $\<\beta^\vee,\alpha\> \leq 0$, it follows from the expression
(\ref{eq:2}) that 
\[\sum_{\mu\in F_2}a_{\mu}\<\beta^{\vee},\mu\> \geq a_{\lambda}\<\beta^\vee,\lambda\>>0.\]
In particular, there exists $\mu \in F_2$ such that $\<\beta^\vee,\mu\> \neq 0$.
Furthermore, whether $\<\alpha^\vee,\lambda\>$ is zero or not, 
by the assumption that $|F(\alpha)| \geq 2$, there exists $\mu \in F
\setminus \{\lambda\}$ such that $\<\alpha^{\vee},\mu\> \neq 0$. 
This finishes the proof. 
\end{proof}

\begin{lemma} \label{lem:tangsr}
Let $\alpha$ be a simple root. Recall that $F(\alpha) = \{\lambda \in F\colon
\<\alpha^{\vee},\lambda\> \neq 0\}$ and put $E(\alpha) := \{ \delta \in E(\wm)
\colon \<\delta,\alpha\> = 1\}$.  Then $\dim (\tg\hs)_{(\alpha)} \leq 1$ and
if $\dim (\tg\hs)_{(\alpha)} = 1$ then 
\begin{enumerate}[(i)]
\item $\alpha \in \Z\wm$; \label{item:simpr1}
\item $|F(\alpha)| \ge 2$; \label{item:simpr2}
\item $\<\delta,\alpha\> \le 1$ for all $\delta \in E(\wm)$;   \label{item:simpr3}
\item $|E(\alpha)| \le 2$;  \label{item:simpr4}
\item If $|E(\alpha)| = 2$ then $E(\alpha) = \{\lambda^\# \in E(\wm)
  \colon \lambda \in F(\alpha)\}.$ \label{item:simpr5}
\end{enumerate}
\end{lemma}
\begin{proof}
Let us assume that $\dim (\tg\hs)_{(\alpha)} \geq 1$.  
Let $[v]$ be a nonzero element of
  $\Vgg_{(\alpha)}$ such that the $G$-equivariant section $s \in H^0(G\cdot x_0,\shN)^G$ defined by
$s(x_0) = [v]$ extends to $X_0$. By Proposition~\ref{prop:elemVggfacts} and  Lemma~\ref{lem:vggalpha}, 
conditions (\ref{item:simpr1}) and (\ref{item:simpr2}) hold. 
Lemma~\ref{lem:codim1orb} and Theorem~\ref{thm:extension-criterion} then imply (\ref{item:simpr3}). 
We now prove (\ref{item:simpr4}). If $|E(\alpha)|\ge 3$, then by 
Theorem~\ref{thm:extension-criterion} and Lemma~\ref{lem:codim1orb}, 
there exist at least three elements $\lambda, \mu,\nu \in F(\alpha)$ such that 
there exist $y_\lambda \in V(\lambda)$, $y_\mu \in V(\mu)$ and $y_\nu \in V(\nu)$ 
for which  $[v]=[y_{\lambda}]=[y_{\mu}]=[y_\nu] \in \Vg$. 
This is impossible by Lemma~\ref{lem:vggalpha} and (\ref{item:simpr4}) is proved. 
We turn to (\ref{item:simpr5}). Suppose $E(\alpha)=\{\lambda^\#, \mu^\#\}$.  
By Lemma~\ref{lem:codim1orb} and Theorem~\ref{thm:extension-criterion},  
there exist $y_\lambda \in V(\lambda)$ and $y_\mu \in V(\mu)$ such that $[v]=[y_{\lambda}]=[y_{\mu}] \in \Vg$. 
Using Lemma~\ref{lem:vggalpha} again,  (\ref{item:simpr5}) follows. 

Finally, we show that $\dim (\tg\hs)_{(\alpha)} \leq 1$. Since $\alpha \in \Z\wm$, there is at least one $\lambda \in E(\alpha)$. Lemma~\ref{lem:codim1orb} and Theorem~\ref{thm:extension-criterion} again imply that $[v]=[y_{\lambda}]$ for some  $y_\lambda \in V(\lambda)$,  which finishes the proof.    
\end{proof}

\begin{remark}
By Corollary~\ref{cor:TX0hs} and the proof of Proposition~\ref{lem:tangsr} below, the preceding lemma gives alternative conditions for $\alpha$ to be N-adapted to $\wm$ when $\wm$ is free. We list them as a separate lemma, since they seem easier to check then those in Corollary~\ref{cor:indiv_Nadapt-spher-roots}. 
\end{remark}

\begin{proof}[Proof of Proposition~\ref{prop:tangsr}]   \label{proofproptangsr}
Lemma~\ref{lem:tangsr} says that  $\dim (\tg\hs)_{(\alpha)} \leq 1$. We assume 
conditions (\ref{item:simpr1}) -- (\ref{item:simpr5}) in Lemma~\ref{lem:tangsr} 
and deduce conditions (\ref{item:inasr1}), (\ref{item:inasr2}),
(\ref{item:inasr4a}), (\ref{item:inasr4b})
and (\ref{item:inasr4c})  in Corollary~\ref{cor:indiv_Nadapt-spher-roots}. 
For (\ref{item:inasr1}) and  (\ref{item:inasr4c}), there is nothing to show. 
For the spherical root $\alpha$,  (\ref{item:inasr2}) follows from  (\ref{item:inasr1}). 
To show (\ref{item:inasr4a}), we first claim that $E(\alpha)$ contains at least one element. Indeed, $\alpha \in \Z\wm$ and $\<\lambda^\#,\alpha\>>0$ for at least one $\lambda\in F$, for otherwise $-\alpha$ would be a dominant weight. The claim now follows from (\ref{item:simpr3}). 
Next, suppose $\lambda^\# \in E(\alpha)$. Clearly $\lambda^\# \in a(\alpha)$. 
We claim that $\alpha^{\vee}-\lambda^\# \neq \lambda^\#$. 
Otherwise, we would have $\lambda^\# = \frac{1}{2}\alpha^{\vee}$, which would contradict (\ref{item:simpr2}). 
This shows $|a(\alpha)|\ge 2$. Now, if $a(\alpha)$ had a third element, then $E(\alpha)$ would have two elements, 
say $\lambda^\#$ and $\mu^\#$, with $\alpha^{\vee}-\lambda^\# \neq \mu^\#$. 
But this yields a contradiction: by (\ref{item:simpr5}), we have that $\<\alpha^{\vee}, \lambda\> = \<\alpha^{\vee},\mu\>=1$ 
and then that  $\alpha^{\vee}-\lambda^\#$ takes the same values as $\mu^\#$ on $F$. 
We have deduced  (\ref{item:inasr4a}). Finally, (\ref{item:inasr4b}) is clear 
since $a(\alpha)=\{\lambda^\#, \alpha^{\vee}-\lambda^\#\}$ for some $\lambda \in F(\alpha)$.   
\end{proof}

\subsection{The non-simple spherical roots} \label{sec:nonsimspher}

To complete the proof of Theorem~\ref{thm:TX0hs}, we show in this section that if
$\gamma$ is a spherically closed spherical root, which is not a simple root and which
occurs as a $\Tad$-weight in $\tg\hs$, then $\gamma$ is N-adapted to
$\wm$. 
 
We recall that conditions (1) and (2) of Corollary~\ref{cor:indiv_Nadapt-spher-roots} follow from Theorem~\ref{thm:vggweightsrefinement}.
We now verify condition (3): if $\delta \in E(\wm)$ such that
  $\<\delta, \gamma\> > 0$ then there exists $\beta \in S\setminus
  S^p(\wm)$ such that $\beta^\vee$ is a positive multiple of
  $\delta$. 
The argument is the same for all the non-simple spherical roots $\gamma$.

Let $v \in V$ be a $\Tad$-eigenvector of weight $\gamma$ such that
 \(
                0 \not=   [v] \in  \Vgg.
\)  
Let $\lambda \in F$. Recall that $z_{\lambda} = x_0 - v_{\lambda}$
and put $a=\<\lambda^\#, \gamma\>$. Assume $a>0$. 

We claim that under this assumption, $\mathrm{codim}_{X_0}G\cdot z_{\lambda}\geq2$.
Indeed, if $\mathrm{codim}_{X_0}G\cdot z_{\lambda}$ were $1$, then by Theorem~\ref{thm:extension-criterion}(B) $a=1$ 
and by Theorem~\ref{thm:extension-criterion}(C) there would exist  $\hat{v}  \in V(\lambda)$  such that  
                         $[v] = [\hat{v}]$  as elements of  $\Vg$.
Therefore, there would exist $\alpha\in S$ such that $\gamma-\alpha\in\pr$,
and such that $X_\alpha\hat{v}$ is nonzero and is equal to $X_{-\gamma+\alpha}x_0$ up to a nonzero scalar multiple.
This would imply $X_{-\gamma+\alpha}v_\lambda\neq0$ and $X_{-\gamma+\alpha}v_\mu=0$ for all $\mu\in F\setminus\{\lambda\}$,
and therefore that there exists $\alpha'\in \sr$ 
such that $\<(\alpha')^\vee,\lambda\>>0$ and $\<(\alpha')^\vee,\mu\>=0$ for all $\mu\in F\setminus\{\lambda\}$, 
which gives a contradiction with Proposition~\ref{prop:codim1orbits} and proves the claim.

The fact that $\mathrm{codim}_{X_0}G\cdot z_{\lambda}\geq2$
means that there exists $\beta\in \sr$ 
such that $\<\beta^\vee,\lambda\>>0$ and $\<\beta^\vee,\mu\>=0$ for all $\mu\in F\setminus\{\lambda\}$. This says exactly that the restriction of $\beta^{\vee}$ to $\Z\wm$ is a positive multiple of $\lambda^\#$, which is condition (3). 

We continue with the remaining conditions of Corollary~\ref{cor:indiv_Nadapt-spher-roots}.
Condition (4) does not apply to non-simple spherical roots.

Condition (5) follows using the analysis of Section~\ref{sec:tad-weights-vgg}. 
Indeed, we have shown that if $[v]$ is a nonzero $\Tad$-eigenvector of weight $2\alpha$ in $\Vgg$, with $\alpha\in S$, 
then $X_\alpha v$ is a (nonzero) scalar multiple of $X_{-\alpha}\hwv$. 
Since $2\alpha \in \Z\wm$, there exists $\lambda \in F$ such that $\<\alpha^{\vee}, \lambda\> >0$ and $\<\lambda^\#, 2\alpha\> > 0$. 
By the argument we used for condition (3), $\lambda$ is the unique element of $F$ which is non-orthogonal to $\alpha$. It follows that 
we actually have that $X_\alpha v$ is a nonzero scalar multiple of $X_{-\alpha}v_{\lambda}$. This implies that the $T$-eigenspace of weight $\lambda-2\alpha$ in $V(\lambda)$ is nonzero,
hence $\<\alpha^\vee,\lambda\>\geq2$. 
Consequently $\<\alpha^\vee,\lambda\>\in\{2,4\}$ and $\<\alpha^\vee,\mu\>=0$ for all $\mu\in F\setminus\{\lambda\}$,
hence $\alpha^\vee$ takes an even value on every element of $\Z\wm$.

Condition (6) follows analogously from Section~\ref{sec:tad-weights-vgg}. 
Indeed, we have shown that if $[v]$ is a nonzero $\Tad$-eigenvector of weight $\alpha+\alpha'$ in $\Vgg$, 
with $\alpha$ and $\alpha'$ orthogonal simple roots,
then $X_\alpha v$, if nonzero, is a scalar multiple of $X_{-\alpha'}\hwv$,
and $X_{\alpha'} v$, if nonzero, is a scalar multiple of $X_{\alpha}\hwv$.

Since $\alpha+\alpha' \in \Z\wm$, there exists $\lambda \in F$ such that $\<\alpha^{\vee}, \lambda\> >0$ and $\<\lambda^\#, \alpha + \alpha'\> > 0$. 
By the argument we used for condition (3), $\lambda$ is the unique element of $F$ which is non-orthogonal to $\alpha$. Then $X_{\alpha} v\neq0$. Indeed if it were $0$, then $X_{\alpha'} v$ would be nonzero, hence scalar multiple of $X_{-\alpha}v_{\lambda}$, which yields a
contradiction:
\[0=X_{\alpha'}X_{\alpha} v=X_{\alpha}X_{\alpha'} v=X_{\alpha}X_{-\alpha} v_{\lambda}\neq0,\]
Therefore $X_{\alpha} v=X_{-\alpha'}\hwv$, and if $\<(\alpha')^\vee, \mu\>\neq0$ then the $T$-eigenspace of weight $\mu-\alpha-\alpha'$ in $V(\mu)$ is nonzero,
hence also $\<\alpha^\vee,\mu\>\neq0$. 
This implies that $\alpha'$ is non-orthogonal to $\lambda$ and orthogonal to $\mu$ for all $\mu\in F\setminus\{\lambda\}$.
Therefore $\alpha^\vee$ and $(\alpha')^\vee$ are equal on every element of $\Z\wm$. This completes the proof of Theorem~\ref{thm:TX0hs}.

\begin{remark}
The information given in this remark is not needed for our results. We include it because it gives explicit conditions on $F$ for each spherically closed spherical root $\gamma$, which is not a simple root, to occur as a $\Tad$-weight in $\tg\hs$, that is, to be N-adapted to $\wm$. 

For each spherically closed spherical root $\gamma$, there exists $\alpha \in \sr$ such that $\<\alpha^\vee, \gamma\> > 0$. If $\gamma$ is a $\Tad$-weight in $\tg\hs$, then $\gamma \in \Z\wm$, and so there exits $\lambda \in F$ such that 
$\<\alpha^\vee, \lambda\> > 0$ and $\<\lambda^\#, \gamma\> >0$. If $\gamma$ is not a simple root, then by the argument above showing that $\gamma$ satifies condition (3) of Corollary~\ref{cor:indiv_Nadapt-spher-roots}, we have that $\lambda$ is the only element of $F$ which is not orthogonal to $\alpha$, that is, $b\lambda^\# = \alpha^{\vee}$ on $\Z\wm$ for some positive integer $b$. 

We now list, for each $\gamma$, the possibilities for $\lambda^\#$. 

\begin{enumerate}

\item If $\gamma = 2\alpha$, with $\alpha$ a simple root, then locally $\gamma = 4\omega$. In this case
 $\alpha^\vee=b\lambda^\#$ with $b\in\{2,4\}$.

\item If $\gamma = \alpha + \alpha'$, with $\alpha$ and $\alpha'$ two orthogonal simple roots, then
locally $\gamma=2\omega+2\omega'$.
In this case $\alpha^\vee=(\alpha')^\vee=b\lambda^\#$ with $b\in\{1,2\}$.

\item If $\gamma = \alpha_1 + \alpha_2 + \ldots
  + \alpha_n$ with support of type $\ssA_n$ with $n\ge 2$, then
locally $\gamma=\omega_1+\omega_n$. In this case, $\alpha^\vee=\lambda^\#$ with $\alpha\in\{\alpha_1,\alpha_n\}$.

\item If $\gamma = \alpha_1 + 2\alpha_2 +
 \alpha_3$ with support of type $\ssA_3$, then
locally $\gamma=2\omega_2$. In this case, we have $\alpha_2^\vee= b\lambda^\#$ with $b\in\{1,2\}$.

\item If $\gamma = \alpha_1 + \ldots
 + \alpha_n$ with support of type $\ssB_n$ with $n \ge 2$, then
locally $\gamma=\omega_1$.
Here $\alpha_1^\vee=\lambda^\#$.

\item If $\gamma =2 \alpha_1 + 2\alpha_2 + \ldots
 +2 \alpha_n$ with support of type $\ssB_n$ with $n \ge 2$, then
locally $\gamma=2\omega_1$.
Here  $\alpha_1^\vee=b\lambda^\#$, with $b\in\{1,2\}$.

\item If $\gamma = \alpha_1 + 2\alpha_2 + 3
  \alpha_3$ with support of type $\ssB_3$, then
locally $\gamma=2\omega_3$.
Here  $\alpha_3^\vee=b\lambda^\#$ with $b\in\{1,2\}$.

\item If $\gamma = \alpha_1 + 2\alpha_2 +
  2\alpha_3 + \ldots+2\alpha_{n-1} + \alpha_n$ with support of type
  $\ssC_n$ with $n\ge 3$, then
locally $\gamma=\omega_2$.
Here $\alpha_2^\vee=\lambda^\#$.

\item If $\gamma = 2\alpha_1 + 2\alpha_2 +
  \ldots + 2\alpha_{n-2} + \alpha_{n-1}+\alpha_{n}$ with support of
  type $\ssD_n$ with $n \ge 4$, then
locally $\gamma=2\omega_1$.
Here $\alpha_1^\vee=b\lambda^\#$ with $b\in\{1,2\}$.

\item If $\gamma = \alpha_1 + 2\alpha_2 +
  3\alpha_3 + 2\alpha_4$ with support of type $\ssF_4$, then
locally $\gamma=\omega_4$.
Here $\alpha_4^\vee=\lambda^\#$.

\item If $\gamma = 4\alpha_1 + 2\alpha_2$ with
 support of type $\ssG_2$, then
locally $\gamma=2\omega_1$.
Here $\alpha_1^\vee=b\lambda^\#$ with $b\in\{1,2\}$. 

\item If $\gamma = \alpha_1 + \alpha_2$ with
  support of type $\ssG_2$, then
locally $\gamma=-\omega_1+\omega_2$.
Here $\alpha_2^\vee=\lambda^\#$.

\end{enumerate}

\end{remark}

\section{The irreducible components of $\ms$} \label{sec:components}

In this section we prove the following.

\begin{theorem} \label{thm:orbitclosaffine}
Let $\wm$ be a free monoid of dominant weights.
Then the $\Tad$-orbit closures in $\ms$, equipped with their reduced induced scheme structure, are affine spaces.
\end{theorem}
The proof is given below. By \cite[Proposition~2.13]{alexeev&brion-modaff} this theorem has the following formal consequence.
\begin{corollary} \label{cor:rootmonfree}
If $X$ is an affine spherical variety with free weight monoid, then its root monoid $\mathscr M_X$ is free too. 
\end{corollary}

Another consequence is that Conjecture~\ref{conj:brion} holds
for free monoids.
\begin{corollary} \label{cor:brionconj}
If $\wm$ is a free monoid of dominant weights then the irreducible
components of $\ms$, equipped with their reduced induced scheme structure, are affine spaces. 
\end{corollary}
\begin{proof}
Since the $\Tad$-orbits in $\ms$ are in bijection with isomorphism
classes of affine spherical $G$-varieties, by \cite[Theorem
1.12]{alexeev&brion-modaff} and there are only finitely many such
isomorphism classes, by \cite[Corollary
3.4]{alexeev&brion-modaff}, we have that every irreducible component $Z$ of $\ms$ contains a dense
$\Tad$-orbit. It then follows from Theorem~\ref{thm:orbitclosaffine} that
$Z$, equipped with its reduced induced scheme structure, is an affine space.
\end{proof}

\begin{proof}[Proof of Theorem~\ref{thm:orbitclosaffine}]
Let $X$ be an affine spherical $G$-variety of weight monoid $\wm$,
seen as a (closed) point in $\ms$.
By \cite[Corollary~2.14]{alexeev&brion-modaff}, we know that the
normalization of $\overline{\Tad\cdot X}$ is an affine space. It is
therefore enough to show that $\overline{\Tad\cdot X}$ is smooth at
$X_0$. We do this by showing that
\begin{equation}\label{eq:smooth}
\dim\tg(\overline{\Tad\cdot X})=\dim\overline{\Tad\cdot X}.
\end{equation}

Recall that $\Sigma^N(X)$ is the basis of the monoid
obtained by saturation of the root monoid $\mathscr M_X$. 
To deduce (\ref{eq:smooth}) we make use of Theorem~\ref{thm:TX0hs}:
the $\Tad$-weights in $\tg(\overline{\Tad\cdot X})\subseteq\tg\ms$ 
are spherical roots N-adapted to $\wm$, each one occurring with multiplicity $1$. 
This, together with the fact that every $\Tad$-weight in $\tg(\overline{\Tad\cdot X})$ 
has to be an element of the root monoid $\mathscr M_X$, and 
hence a nonnegative integer linear combination of elements of $\Sigma^N(X)$,
gives (\ref{eq:smooth}) once we prove Proposition~\ref{prop:cut}
below. Indeed, applying this proposition with $\Sigma = \Sigma^N(X)$ yields
that the $\Tad$-weights in 
$\tg(\overline{\Tad\cdot X})$ belong to $\Sigma^N(X)$, while $\dim
\overline{\Tad\cdot X} = |\Sigma^N(X)|$ by \cite[Proposition
2.13]{alexeev&brion-modaff}. 
\end{proof}

\begin{proposition}\label{prop:cut}
Let $\Sigma$ be a subset of $\Sigma^{sc}(G)$ such that every $\gamma \in
\Sigma$ is N-adapted to $\wm$. If $\sigma\in\Sigma^{sc}(G)\cap\N\Sigma$ is N-adapted to $\wm$,
then $\sigma\in\Sigma$.
\end{proposition}
\begin{proof}
First of all, $\sigma$ (of spherically closed type) must be compatible
with $S^p(\wm)$ and is 
a nonnegative integer linear combination of other elements of
$\Sigma^{sc}(G)$ that satisfy the same compatibility condition.  This gives strong restrictions.
Indeed, $\sigma$ can only be the sum of two simple roots (equal or not, orthogonal or not).
All the other types of spherical roots have support that nontrivially intersects $S^p(\wm)$, 
and they can be excluded by a straighforward if somewhat lengthy case-by-case verification.

Moreover, $\sigma$ cannot be the double of a simple root, say $2\alpha$, with $\alpha\in\Sigma$,
since $\alpha$ and $2\alpha$ cannot both be N-adapted to $\wm$.
Indeed, if $2\alpha$ is N-adapted to $\wm$ then, since $\<\alpha^\vee, 2\alpha\> >0$ and $\alpha^\vee \in \wm^\vee$, there exists $\delta \in E(\wm)$ such that $\<\delta, 2\alpha\> >0$. Condition (\ref{item:inasr3}) of Corollary~\ref{cor:indiv_Nadapt-spher-roots} tells us that  
$\alpha^\vee$ is a positive multiple of $\delta$. By condition (\ref{item:inasr5}) of the same corollary, $\alpha^\vee$ is not primitive in $(\Z\wm)^*$. If now $\alpha\in \Z\wm$, then it follows from $\<\alpha^\vee, \alpha\>=2$ that  $\alpha^\vee=2\delta$ on $\Z\wm$.
Hence $\delta$ is the only element of $a(\alpha)$ and $\alpha$ is not N-adapted
to $\wm$.  

Analogously, $\sigma$ cannot be the sum of two orthogonal simple roots, say $\alpha+\alpha'$, 
with $\alpha$ and $\alpha'$ in $\Sigma$.
Indeed, since $\alpha+\alpha'$ is adapted to $\wm$ and $\<\alpha^\vee,\alpha\>\neq\<(\alpha')^\vee,\alpha\>$,
$\alpha$ cannot belong to $\Z\wm$.

Finally, let $\sigma$ be the sum of two nonorthogonal simple roots, say $\alpha_1+\alpha_2$,
with $\alpha_1$ and $\alpha_2$ in $\Sigma$.
 Take $\delta\in E(\wm)$ with $\<\delta,\sigma\> >0$. 
Such a $\delta$ exists because $\<\alpha_1^{\vee}, \sigma\>$ or $\<\alpha_2^\vee, \sigma\>$ is positive, 
$\sigma \in \Z\wm$ and $\wm \inn \dw$.  
Then $\delta$ must be positive on at least one of the two simple roots $\alpha_1$ or $\alpha_2$. 
Suppose it is positive on $\alpha_1$.  
Then $\delta \in a(\alpha_1)$, since $\alpha_1$ is N-adapted to $\wm$, 
hence $\delta$ takes the value $1$ on $\alpha_1$.  
By condition (\ref{item:inasr3}) of Corollary~\ref{cor:indiv_Nadapt-spher-roots} it follows that  $\alpha_1^\vee=2\delta$, 
which is not possible if $\alpha_1$ is N-adapted to $\wm$.
\end{proof}

\begin{remark}\label{rem:alper}
While the reduced induced scheme structure is the only natural scheme structure on the $\Tad$-orbit closures of Theorem~\ref{thm:orbitclosaffine}, there is at least one other natural scheme structure on the irreducible components of $\ms$, namely the one given by the primary ideals of $\k[\ms]$ associated to minimal primes. One can ask whether Conjecture~\ref{conj:brion} remains true for that scheme structure. Another natural question is whether or when $\ms$ is in fact a reduced scheme. We note that the tangent space $\tg\ms$ might fail to detect the "non-reducedness" of $\ms$. For example, the two affine schemes $\Spec(\k[x,y]/\<xy\>)$ and $\Spec(\k[x,y]/\<x^2y\>)$ have the same tangent space at the point corresponding to the maximal ideal $\<x,y\>$.  
\end{remark}

\def\cprime{$'$}
\providecommand{\bysame}{\leavevmode\hbox to3em{\hrulefill}\thinspace}
\providecommand{\MR}{\relax\ifhmode\unskip\space\fi MR }
% \MRhref is called by the amsart/book/proc definition of \MR.
\providecommand{\MRhref}[2]{%
  \href{http://www.ams.org/mathscinet-getitem?mr=#1}{#2}
}
\providecommand{\href}[2]{#2}

\end{document}